\documentclass[a4paper]{amsart}
\usepackage{graphicx}
\usepackage{amsmath}
\usepackage{amssymb}
\usepackage{oldgerm}
\usepackage{mathdots}
\usepackage{stmaryrd}
\usepackage{bm}
\usepackage[all]{xy}
\usepackage{color}
\usepackage{enumerate}
\usepackage{fourier}
\usepackage{hyperref}
\usepackage{mathrsfs}
\usepackage{tikz}
\usetikzlibrary{arrows,%
  decorations.markings,%
  matrix,%
  shapes}

\newcommand{\Hom}{\operatorname{Hom}\nolimits}
\newcommand{\Span}{\operatorname{span}\nolimits}
\renewcommand{\Im}{\operatorname{Im}\nolimits}

\newcommand{\sthom}{\operatorname{\underline{Hom}}\nolimits}

\newcommand{\Ker}{\operatorname{Ker}\nolimits}
\newcommand{\Top}{\operatorname{top}\nolimits}
\newcommand{\Soc}{\operatorname{soc}\nolimits}

\newcommand{\Ext}{\operatorname{Ext}\nolimits}

\newcommand{\HH}{\operatorname{HH}\nolimits}

\newcommand{\rank}{\operatorname{rank}\nolimits}

\newcommand{\ra}{\mathfrak{r}}

\newcommand{\cx}{\operatorname{cx}\nolimits}

\newcommand{\V}{\operatorname{V}\nolimits}

\newcommand{\Char}{\operatorname{char}\nolimits}

\newtheorem{theorem}{Theorem}[section]

\newtheorem{corollary}[theorem]{Corollary}

\newtheorem{lemma}[theorem]{Lemma}
\newtheorem{proposition}[theorem]{Proposition}

\theoremstyle{definition}
\newtheorem*{definition}{Definition}
\theoremstyle{definition}

\newtheorem*{questions}{Questions}
\newtheorem*{example}{Example}
\theoremstyle{definition}

\newtheorem{remark}[theorem]{Remark}
\theoremstyle{remark}

\theoremstyle{definition}

\theoremstyle{definition}

\theoremstyle{definition}

\begin{document}

\title{Modules of constant Jordan type over quantum complete intersections}

\author{Petter Andreas Bergh, Karin Erdmann and David A.\ Jorgensen}

\address{Petter Andreas Bergh \\ Institutt for matematiske fag \\
NTNU \\ N-7491 Trondheim \\ Norway} \email{petter.bergh@ntnu.no}
\address{Karin Erdmann \\ Mathematical Institute \\ University of Oxford \\ Andrew Wiles Building \\ Radcliffe Observatory Quarter \\ Woodstock Road \\ Oxford \\ OX2 6GG \\ United Kingdom}
\email{erdmann@maths.ox.ac.uk}
\address{David A.\ Jorgensen \\ Department of Mathematics \\ University
of Texas at Arlington \\ Arlington \\ TX 76019 \\ USA}
\email{djorgens@uta.edu}




\begin{abstract} We initiate the study of modules
	of constant Jordan type for quantum complete
	intersections, and prove a range of basic properties. We then show that
	for these algebras, constant Jordan type is an invariant of Auslander-Reiten components.
	Finally, we classify modules with stable constant Jordan type
	$[1]$ or $[n-1]$ in the 2-generator case.
\end{abstract}

\maketitle

\section{Introduction}\label{sec:intro}

Varieties for modular group representations  were introduced and studied by Carlson and others in order to understand modules without having to
rely on complete  classifications. Many aspects in this approach are controlled by elementary abelian $p$-groups,  making use of the fact that their group algebras
are truncated polynomial algebras.  In particular,
the  concept of rank variety has proved to be extremely powerful. These varieties control projectivity of finite-dimensional modules, and this is known as ``Dade's Lemma.''

\sloppy Modules of constant Jordan type were introduced by Carlson, Friedlander and Pevtsova in \cite{CarlsonFriedlanderPevtsova}. By using the idea of $\pi-$points, they obtained versions not only for finite groups, but more generally for finite group schemes; see the book  \cite{Benson} for extensive discussions. It turns out that a substantial part of the theory relies on Dade's Lemma. Around the same time, in \cite{BensonErdmannHolloway}, rank varieties were introduced for certain $q$-complete intersections. It was proved that
the analog of Dade's Lemma holds, namely that  rank varieties control projectivity also in this case. It is therefore natural to study  modules of constant Jordan type in this setting. 

The aim of this paper is to start this investigation. We introduce the relevant algebras, and prove a  range of general properties. For example, we show that constant Jordan type is preserved under taking direct sums and summands, and taking syzygies and cosyzygies.  We prove that certain constant Jordan types are unachievable by modules.  We also address the question of how abundant are modules of constant Jordan type.

We also show that in our setting, constant Jordan type is preserved under Auslander-Reiten translation, and moreover that constant Jordan type is an invariant of Auslander-Reiten components. 

For the group algebra setting, it is known 
that modules with stable constant Jordan type $[1]$ or $[p-1]$ are precisely the syzygies of the
trivial module. One expects that this might generalize.  In the last section, we prove a result in this direction.
Namely, for the 2-generator case, with certain assumptions, we 
show that the modules  with stable constant Jordan type
$[1]$ or $[n-1]$ are precisely the syzygies of the simple module.

For group representations and group schemes in general, the tools which were used are mostly based on the fact that the rings are Hopf algebras.
This is not the case in our setting, with quantum complete intersections. This means that a lot of the machinery and techniques from the group scheme setting are not available, and consequently the proofs are more difficult. As well, it is not clear whether all the results in the group scheme setting actually generalize.

\section{Constant Jordan type}\label{sec:jordantype}

Let $k$ be a field, $n \ge 2$ a positive integer, and define $n'$ by
$$n' = \left \{ 
\begin{array}{ll}
n & \text{if } \Char k = 0 \\
n/ \gcd (n, \Char k) & \text{if } \Char k > 0
\end{array} \right.$$
Furthermore, let $q \in k$ be a primitive $n'$th root of unity, and fix a positive integer $c \ge 2$. The algebra we shall consider is the quantum complete intersection  
$$A^c_q = k \langle x_1, \dots, x_c \rangle / \left ( x_i^n, x_ix_j-qx_jx_i (i<j) \right )$$
which is a finite dimensional selfinjective algebra of dimension $n^c$. 

Quantum complete intersections can be defined more generally, but the essential thing for us is that any linear form in the variables $x_1, \dots, x_c$ is $n$-nilpotent. Namely, by \cite[Lemma 2.3]{BensonErdmannHolloway}, given any nonzero $c$-tuple $\lambda = ( \lambda_1, \dots, \lambda_c) \in k^c$, the element 
$$u_{\lambda} = \lambda_1x_1 + \cdots + \lambda_cx_c$$
in $A^c_q$ satisfies $u_{\lambda}^n = 0$ (and $n$ is the smallest such power with $u_{\lambda}^n = 0$). Consequently, the subalgebra $k[u_{\lambda}]$ of $A^c_q$ generated by $u_{\lambda}$ is isomorphic to the $n$-dimensional truncated polynomial ring $k[x]/(x^n)$. Note that $A^c_q$ is free both as a left and as a right module for this subalgebra.

\sloppy Up to isomorphism, the truncated polynomial ring $k[x]/(x^n)$ admits $n$ finitely generated indecomposable left modules, namely $k[x]/(x^i)$ for $1 \le i \le n$. Since $\dim_k k[x]/(x^i) = i$, these are uniquely determined by their dimensions. Now let $M$ be an $A^c_q$-module. When we restrict the module to the subalgebra $k[u_{\lambda}]$, it decomposes into a direct sum of the $n$ indecomposable $k[u_{\lambda}]$-modules.

\begin{definition}\label{def:jordantype}
Fix an $A^c_q$-module $M$.

(1) Let $\lambda$ be a nonzero $c$-tuple in $k^c$, and $M_i$ the indecomposable $k[u_{\lambda}]$-module with $\dim_k M_i = i$. Then the \emph{Jordan type} of $M$ with respect to $\lambda$ is
$$[1]^{d_1} [2]^{d_2} \cdots [n]^{d_n}$$
if $M$ as a $k[u_{\lambda}]$-module decomposes as a direct sum $M \simeq M_1^{d_1} \oplus \cdots \oplus M_n^{d_n}$. In this case, the \emph{stable} Jordan type of $M$ with respect to $\lambda$ is 
$$[1]^{d_1} \cdots [n-1]^{d_{n-1}}$$

(2) The module $M$ has \emph{constant Jordan type} if its Jordan type is the same for all nonzero $\lambda \in k^c$.
\end{definition}

The phrase \emph{Jordan type} refers to the action of the truncated polynomial ring $k[x]/(x^n)$ on its $n$ indecomposable modules $M_1, \dots, M_n$, where $M_i = k[x]/(x^i)$. For $M_i$ has a $k$-vector space basis $w_1, \dots, w_i$, with $xw_j = w_{j+1}$ for $1 \le j \le i-1$, and so the action of $x$ on $M_i$ can be defined in terms of the transpose of the Jordan block
$$\left ( \begin{array}{cccccc}
0 & 1 & 0 & \cdots & 0 \\
0 & 0 & 1 & \cdots & 0 \\
\vdots & \vdots & \vdots & \ddots & \vdots \\
0 & 0 & 0 & \cdots & 1 \\
0 & 0 & 0 & \cdots & 0 
\end{array} \right )$$
of size $i$ and with eigenvalue $0$. Thus if $M$ has Jordan type $[1]^{d_1} \cdots [n]^{d_n}$ with respect to $\lambda = ( \lambda_1, \dots, \lambda_c)$, then the action of $u_{\lambda} =  \lambda_1x_1 + \cdots + \lambda_cx_c$ on $M$ is given by the transpose of a Jordan matrix having $d_i$ such Jordan blocks of size $i$ for each $1 \le i \le n$.

To simplify the notation, it is convenient to write Jordan types in such a way that only the Jordan blocks that are actually involved appear. That is, if $M$ is a nonzero $A^c_q$-module and $\lambda$ a nonzero $c$-tuple, then there are integers $1 \le a_1 < \cdots < a_t \le n$ such that $M \simeq M_{a_1}^{d_{a_1}} \oplus \cdots \oplus M_{a_t}^{d_{a_t}}$ when restricted to $k[u_{\lambda}]$ (where $M_i$ is the indecomposable $k[u_{\lambda}]$-module of dimension $i$ and all the exponents $d_{a_i}$ are nonzero). We shall then in most cases write the Jordan type of $M$ with respect to $\lambda$ as
$$[a_1]^{d_{a_1}} \cdots [a_t]^{d_{a_t}}$$
instead of having to write
$$[1]^{d_1} \cdots [n]^{d_n}$$
with $d_i =0$ for $i \notin \{ a_1, \dots, a_t \}$. Similarly we apply this short-hand notation also for the stable Jordan type. Furthermore, when $d_i =1$ we just write $[i]$ instead of $[i]^{d_i}$.

\begin{example}
Two trivial examples of indecomposable $A^c_q$-modules of constant Jordan type are $k$ and $A^c_q$. For $k$, the Jordan type with respect to any nonzero $\lambda$ is clearly $[1]$, whereas for $A^c_q$ it is $[n]^{n^{c-1}}$ since $A^c_q$ is free of rank $n^{c-1}$ over $k[u_{\lambda}]$.

More generally, denote the radical of $A^c_q$ by $\ra$. The Loewy length of $A^c_q$ is $(n-1)^c + 1$, and we claim that for every pair of integers $0 \le s < t \le (n-1)^c + 1$, the $A^c_q$-module $\ra^s / \ra^t$ has constant Jordan type (but may not be indecomposable, for example when $t = s+1$). To see this, note that as a $k$-vector space, the module $\ra^s / \ra^t$ has a basis
$$\left \{ x_1^{e_1}x_2^{e_2} \cdots x_c^{e_c} \mid s \le e_1 + \cdots + e_c \le t \right \}$$
Now let $\lambda = ( \lambda_1, \dots, \lambda_c)$ be a nonzero $c$-tuple in $k^c$. By renumbering the generators $x_1, \dots, x_c$ if necessary, we may suppose that $\lambda_1$ is nonzero. Using now that
$$x_1 =  \lambda^{-1} \left ( u_{\lambda} - \sum_{i=2}^c \lambda_i x_i \right )$$
it is straightforward to check that
$$\left \{ u_{\lambda}^{e_1}x_2^{e_2} \cdots x_c^{e_c} \mid s \le e_1 + \cdots + e_c \le t \right \}$$
is another basis for the module $\ra^s / \ra^t$. Consequently, the module structure of $\ra^s / \ra^t$ when restricted to $k[u_{\lambda}]$ is the same as when restricted to $k[x_1]$. This shows that $\ra^s / \ra^t$ has constant Jordan type.

As a specific example, take the indecomposable module $A^c_q / \ra^3$, and assume first that $n \ge 3$ so that $x_i^2 \ne 0$. The monomials
$$\xymatrix@C=5pt@R=5pt{
& & 1 \\
& x_1 & \cdots & x_c \\
x_1^2 & x_1x_2 & \cdots & x_1x_c & x_ix_j
}$$
form a basis for the module, where $2 \le i \le j \le c$. Over the subalgebra $k[x_1]$, the subset $\{ 1, x_1, x_1^2 \}$ is a basis for the indecomposable module of dimension $3$, and the $c-1$ subsets $\{ x_i, x_1x_i \}$ for $2 \le i \le c$ form bases for indecomposable modules of dimension $2$. Finally, the $c(c-1)/2$ subsets $\{ x_ix_j \}$ for $2 \le i \le j \le c$ form bases for indecomposable modules of dimension $1$. Thus the module $A^c_q / \ra^3$ has constant Jordan type
$$[1]^{c(c-1)/2} [2]^{(c-1)} [3]$$
Note that if $n$ were $2$, then there would be no indecomposable $k[x_1]$-module of dimension $3$, but instead one more of dimension $2$. Also, there would be $c-1$ less indecomposable $k[x_1]$-module of dimension $1$. Therefore, in this case, the constant Jordan type of $A^c_q / \ra^3$ would be
$$[1]^{(c^2-3c+2)/2} [2]^{c}$$
\end{example}

The following result records some elementary properties on modules of constant Jordan type

\begin{proposition}\label{prop:elementary}
(1) If $M$ is an $A^c_q$-module of constant Jordan type $[1]^{d_1} \cdots [n]^{d_n}$, then $\dim_k M = \sum_{i=1}^n id_i$. Moreover, the dual $\Hom_k(M,k)$ has the same constant Jordan type as a right module.

(2) If $M$ and $N$ are $A^c_q$-modules of constant Jordan types $[1]^{d_1} \cdots [n]^{d_n}$ and $[1]^{e_1} \cdots [n]^{e_n}$, respectively, then the direct sum $M \oplus N$ has constant Jordan type $[1]^{d_1+e_1} \cdots [n]^{d_n+e_n}$.
\end{proposition}

\begin{proof}
For the dual module, note that when dualizing an indecomposable module over $k[x]/(x^n)$, the result is the same indecomposable module. 
\end{proof}

The modules of the form $\ra^s / \ra^t$ that we looked at in the example are well behaved and easy to deal with. In general, there does not seem to exist any effective method for determining whether or not an arbitrary indecomposable module has constant Jordan type. However, it turns out that such modules can be described in terms of certain open subsets of the affine space $k^c$, where we use the Zariski topology. We shall use this to prove the converse of the second part of Proposition \ref{prop:elementary}, namely that the direct summand of a module of constant Jordan type also has constant Jordan type. An alternative method would be to adapt the arguments from \cite[Sections 4.5 and 5.1]{Benson}, which are based on Carlson, Friedlander and Pevtsova's original arguments from \cite{CarlsonFriedlanderPevtsova}.

\begin{definition}\label{def:opensets}
For an $A^c_q$-module $M$ and integer $1 \le i \le n-1$, define
$$\mathcal U^i_M = \left \{ \lambda \in k^c \setminus \{ 0 \} \mid \rank (M \xrightarrow{u_{\lambda}^i} M) \ge \rank (M \xrightarrow{u_{\sigma}^i} M) \text{ for all } \sigma \in k^c \right \}$$
\end{definition}

In other words, the set $\mathcal U_M^i$ is the collection of all nonzero $\lambda$ for which the linear operator $u_{\lambda}^i$ on $M$ has maximal rank. Note that this is a nonempty set by definition. Recall that for a truncated polynomial algebra $k[x]/(x^n)$, the rank of $x^i$ as a linear operator on the $t$-dimensional module $M_t = k[x]/(x^t)$ is $\max \{ 0, t-i \}$. Consequently, if the Jordan type of $M$ with respect to $\lambda$ is $[1]^{d_1} [2]^{d_2} \cdots [n]^{d_n}$, that is, if $M$ as a $k[u_{\lambda}]$-module decomposes as $M_1^{d_1} \oplus \cdots \oplus M_n^{d_n}$, then the rank of the linear operator $u_{\lambda}^i$ on $M$ is 
$$d_{i+1} + 2d_{i+2} + \cdots + (n-i)d_n$$
provided $i \le n-1$. We use this elementary fact in the proof of the following lemma, which shows that the sets we have just defined describe the modules of constant Jordan type. 

\begin{lemma}\label{lem:characterizeopen}
Let $M$ be an $A^c_q$-module.

(1) If $\lambda$ and $\sigma$ belong to $\cap_{i=1}^{n-1} \mathcal U_M^i$, then the Jordan types of $M$ with respect to $\lambda$ and $\sigma$ are the same.

(2) The module $M$ has constant Jordan type if and only if $\cap_{i=1}^{n-1} \mathcal U_M^i = k^c \setminus \{ 0 \}$.
\end{lemma}

\begin{proof}
To prove (1), suppose that $M$ decomposes as $M_1^{d_1} \oplus \cdots \oplus M_n^{d_n}$ over $k[u_{\lambda}]$, and as $M_1^{e_1} \oplus \cdots \oplus M_n^{e_n}$ over $k[u_{\sigma}]$. Since both $\lambda$ and $\sigma$ belong to $\cap_{i=1}^{n-1} \mathcal U_M^i$, we see that for each $i$, the rank of the linear operators $u_{\lambda}^i$  and $u_{\sigma}^i$ on $M$ are the same. It therefore follows from the discussion preceding the lemma that there are equalities
\begin{eqnarray*}
d_2 + 2d_3 + \cdots + (n-1)d_n & = & e_2 + 2e_3 + \cdots + (n-1)e_n \\
d_3 + 2d_4 + \cdots + (n-2)d_n & = & e_3 + 2e_4 + \cdots + (n-2)e_n \\
& \vdots & \\
d_{n-1} + 2d_n & = & e_{n-1} + 2e_n \\
d_n & = & e_n
\end{eqnarray*}
This gives $d_i = e_i$ for $2 \le i \le n$, and in turn also $d_1 = e_1$ by considering dimensions. 

For (2), note that if $M$ has constant Jordan type then trivially $\mathcal U_M^i = k^c \setminus \{ 0 \}$ for all $i$. The converse is an immediate consequence of (1).
\end{proof}

The following lemma shows that the maximal rank sets that we have defined are open subsets of affine $c$-space. Moreover, when the ground field is infinite, then the sets corresponding to a direct sum of modules is the intersection of the sets corresponding to the summands.

\begin{lemma}\label{lem:opensets}
For every $A^c_q$-module $M$ and integer $1 \le i \le n-1$, the set $\mathcal U_M^i$ is open in $k^c$. Moreover, if the ground field $k$ is infinite, then $\mathcal U_{M \oplus N}^i = \mathcal U_M^i \cap \mathcal U_N^i$ for every $A^c_q$-module $N$.
\end{lemma}

\begin{proof}
An argument similar to the proof of \cite[Lemma 9]{Oppermann} shows that the sets $\mathcal U_M^i$ are open. If $k$ is infinite, then the intersection of two nonempty open sets in $k^c$ is always nonempty, in particular $\mathcal U_M^i \cap \mathcal U_N^i \neq \emptyset$. Since the rank of $u_{\lambda}^i$ on $M \oplus N$ is the sum of the ranks on $M$ and $N$, this proves that $\mathcal U_{M \oplus N}^i = \mathcal U_M^i \cap \mathcal U_N^i$.
\end{proof}

We can now show that the converse of the second part of Proposition \ref{prop:elementary} also holds. 

\begin{corollary}\label{cor:directsum}
If the field $k$ is infinite, then for every pair $M,N$ of $A^c_q$-modules, the direct sum $M \oplus N$ has constant Jordan type if and only if both $M$ and $N$ have.
\end{corollary}

\begin{proof}
If $M \oplus N$ has constant Jordan type, then from Lemma \ref{lem:characterizeopen} we see that $\mathcal U_{M \oplus N}^i = k^c \setminus \{ 0 \}$ for all $i$. By Lemma \ref{lem:opensets}, this implies that both $\mathcal U_M^i$ and $\mathcal U_N^i$ equal $k^c \setminus \{ 0 \}$ for all $i$, and so by Lemma \ref{lem:characterizeopen} again both $M$ and $N$ have constant Jordan type.
\end{proof}

Having seen some examples and elementary properties of modules of constant Jordan type, it is natural to ask the following basic questions:

\begin{questions}
(1) Which indecomposable $A^c_q$-modules have constant Jordan type?

(2) Which sequences $(a_1, \dots, a_n)$ in $\mathbb{Z}_+^n$ occur as the (exponents of the) Jordan type $[1]^{a_1} \cdots [n]^{a_n}$ for some indecomposable $A^c_q$-module of constant Jordan type?

(3) Which sequences $(a_1, \dots, a_{n-1})$ in $\mathbb{Z}_+^{n-1}$ occur as the (exponents of the) stable Jordan type $[1]^{a_1} \cdots [n-1]^{a_{n-1}}$ for some indecomposable $A^c_q$-module of constant Jordan type?
\end{questions}

Regarding the second question, the following result shows that the sequences
$$(0,1,0 \dots, 0), (0,0,1,0, \dots, 0), \dots, (0, \dots, 0,1)$$
in $\mathbb{Z}_+^n$ are not the exponents of the Jordan types of modules of constant Jordan type.

\begin{proposition}\label{prop:nottype}
For $2 \le a \le n$, there does not exist an $A^c_q$-module of constant Jordan type $[a]$.
\end{proposition}

\begin{proof}
Suppose such a module $M$ exists, and denote the radical of $A^c_q$ by $\ra$. For every nonzero $\lambda \in k^c$, the restriction of $M$ to $k[u_{\lambda}]$ is isomorphic to the indecomposable module $k[u_{\lambda}] / ( u_{\lambda}^a )$ of dimension $a$. Therefore there are strict inclusions
$$M \supset u_{\lambda}M \supset \cdots \supset u_{\lambda} ^{a-1}M \supset 0$$
with one-dimensional quotients. Now consider the radical filtration 
$$M \supseteq \ra M \supseteq \cdots \supseteq \ra^{a-1} M \supseteq \cdots$$
of $M$ as an $A^c_q$-module. If $\ra^i M = \ra^{i+1} M$, then $\ra^i M = 0$ by Nakayama's lemma. Moreover, since $u_{\lambda} \in \ra$, the inclusion $u_{\lambda}^i M \subseteq \ra^i M$ holds for all $i$. Combining all this with the fact that $\dim_k M = a$, we see that the radical series of $M$ is
$$M \supset \ra M \supset \cdots \supset \ra^{a-1} M \supset 0$$
with one-dimensional quotients. Thus $M$ has a $k$-vector space basis $\{ m_1, \dots, m_a \}$ with $m_i \in \ra^{i-1}M \setminus \ra^i M$ for all $i$. 

From the above, for every nonzero $\lambda \in k^c$ there is a nonzero element $\alpha \in k$ such that $u_{\lambda} m_1 = \alpha m_2$. In particular, there are nonzero elements $\alpha_1$ and $\alpha_2$ such that $x_1 m_1 = \alpha_1 m_2$ and $x_2 m_1 = \alpha_2 m_2$. But then
$$( \alpha_2x_1 - \alpha_1 x_2 ) m_1 = 0$$
which means that $u_{\lambda} m_1 = 0$ for $\lambda = ( \alpha_2, \alpha_1, 0, \dots, 0)$. This is a contradiction.
\end{proof}

Next, we look at modules of constant Jordan type from a homological point of view. For a finite-dimensional algebra $A$, every module $M$ has a minimal projective resolution
$$\cdots \to P_2 \xrightarrow{\partial_2} P_1 \xrightarrow{\partial_1} P_0 \xrightarrow{\partial_0} M \to 0$$
with $\Im \partial_i \subseteq \ra_A P_{i-1}$. The \emph{complexity} of $M$ is defined as
$$\cx M = \inf \{ m \ge 0 \mid \text{there exists } b \in \mathbb{R} \text{ with } \dim_k P_t \le bt^{m-1} \text{ for all } t \ge 0 \}$$
The complexity of a module might be infinite, and is at most the maximal complexity obtained by the simple $A$-modules. For our quantum complete intersection $A^c_q$, it follows from \cite[Theorem 5.3]{BerghOppermann} that the complexity of the simple module $k$ is $c$, since it equals the rate of growth of $\Ext_{A^c_q}^*(k,k)$. Consequently, the complexity of every $A^c_q$-module is at most $c$. Moreover, for every integer $0 \le m \le c$, there exists an $A^c_q$-module having complexity $m$. Namely, by \cite[Theorem 5.5]{BerghOppermann}, the algebra $A^c_q$ has finitely generated cohomology, and the claim now follows from \cite[Theorem 2.5(c) and Theorem 4.4]{EHSST}.

The following result shows that the non-free $A^c_q$-modules of constant Jordan type must have maximal complexity. Moreover, it shows that the property of having constant Jordan type is preserved under syzygies and cosyzygies.

\begin{theorem}\label{thm:complexitysyzygies}
(1) If $M$ is an $A^c_q$-module of constant Jordan type, then either $M$ is free, or $\cx M = c$.

(2) Given a short exact sequence
$$0 \to M \to F \to N \to 0$$
of $A^c_q$-modules with $F$ a free module, the module $N$ has constant Jordan type if and only if $M$ does. In fact, if $N$ has constant Jordan type $[1]^{d_1} \cdots [n]^{d_n}$, then $M$ has constant Jordan type $[1]^{d_{n-1}} \cdots [n-1]^{d_1} [n]^d$, where $d = rn^{c-1} - \left ( d_1 + \cdots + d_n \right )$ and $r$ is the rank of the free $A^c_q$-module $F$. Conversely, if $M$ has constant Jordan type $[1]^{d_1} \cdots [n]^{d_n}$, then $N$ has constant Jordan type $[1]^{d_{n-1}} \cdots [n-1]^{d_1} [n]^d$.
\end{theorem}

\begin{proof}
(1) Suppose that $M$ is a non-free module of constant Jordan type. By \cite[Theorem 2.6]{BensonErdmannHolloway}, there exists a nonzero $c$-tuple $\lambda \in k^c$ with the property that $M$ is not free as a $k[u_{\lambda}]$-module. As $M$ has constant Jordan type, the same must hold for every nonzero $\lambda \in k^c$. Thus the rank variety $\V_{A^c_q}^r(M)$ of $M$, as defined in \cite{BensonErdmannHolloway}, must equal $k^c$. By \cite[Corollary 3.7]{BerghErdmann}, the dimension of the rank variety of a module equals its complexity, hence $\cx M = c$.

(2) As before, let $M_1, \dots, M_n$ be the indecomposable $k[x]/(x^n)$-modules, with $M_i = k[x]/(x^i)$. For $1 \le i \le n-1$, there are short exact sequences
$$0 \to M_{n-i} \to k[x]/(x^n) \to M_i \to 0$$
and so $\Omega_{k[x]/(x^n)}^1(M_i) \simeq M_{n-i} \simeq \Omega_{k[x]/(x^n)}^{-1}(M_i)$. Now suppose that $N$ has Jordan type $[1]^{d_1} \cdots [n]^{d_n}$ with respect to $\lambda$. Since the module $F$ is free as a $k[u_{\lambda}]$-module, the module $M$ is isomorphic to $\Omega_{k[x]/(x^n)}^1(N) \oplus Q$ over $k[u_{\lambda}]$, where $Q$ is some free $k[u_{\lambda}]$-module. Thus the Jordan type of $M$ with respect to $\lambda$ must be
$$[1]^{d_{n-1}} \cdots [n-1]^{d_1} [n]^d$$
for some $d \ge 0$. Comparing dimensions we obtain
\begin{eqnarray*}
d_{n-1} + 2d_{n-2} + \cdots + (n-1)d_1 + nd & = & \dim_k M \\
& = & \dim_k F - \dim_k N \\
& = & r n^c - \left ( d_1 + 2d_2 + \cdots + nd_n \right )
\end{eqnarray*}
which in turn gives
$$d = rn^{c-1} - \left ( d_1 + \cdots + d_n \right )$$
The converse is proved exactly the same way.
\end{proof}

\begin{corollary}\label{cor:syzygies}
Let $i \in \mathbb{Z}$ be any integer. An $A^c_q$-module $M$ has constant Jordan type if and only if $\Omega_{A^c_q}^i(M)$ does. Moreover, if $M$ has constant stable Jordan type $[1]^{d_1} \cdots [n-1]^{d_{n-1}}$, then so does $\Omega_{A^c_q}^{2i}(M)$, whereas $\Omega_{A^c_q}^{2i+1}(M)$ has constant stable Jordan type $[1]^{d_{n-1}} \cdots [n-1]^{d_1}$.
\end{corollary}

A reasonable question to ask is how abundant are modules of constant Jordan type.  The following discussion offers an answer to this question.

Let $M$ be a finitely generated $A_q^c$-module of dimension $d$.  We fix a $k$-basis $B$ of $M$.  For 
$\lambda=(\lambda_1,\dots,\lambda_c)\in k^c$ and $u_\lambda=\lambda_1x_1+\cdots+\lambda_cx_c$ we let $[u_\lambda]_B$ denote the matrix representing the linear operator $u_\lambda:M\to M$ with respect to $B$.  In Particular, $[x_i]_B$ represents the linear operator $x_i:M\to M$ with respect to the basis $B$, for $1\le i\le c$.
The matrices $[x_i]_B$ determine the $A_q^c$-module $M$ in the following sense. Let $N$ be another $d$-dimensional $A_q^c$-module. Then $M$ and $N$ are isomorphic as $A_q^c$-modules, with isomorphism 
$\alpha:M\to N$, if and only if there exists an invertible $d\times d$ matrix $E$ such that $E[x_i]_B=[x_i]_{\alpha(B)}E$ for all $1\le i\le c$, where $[u_\lambda]_{\alpha(B)}$ is the matrix of the linear operator $u_\lambda:N\to N$ with respect to the basis $\alpha(B)$. In this case, $E$ represents the $A_q^c$-linear isomorphism $\alpha$, as a $k$-linear map, with respect to the bases $B$ and $\alpha(B)$. Note that we also have $E[u_\lambda]_B=[u_\lambda]_{\alpha(B)}E$ for all $\lambda$,  

We want to associate $A_q^c$-modules to points in a certain affine space.  We will do this using the matrices 
$[x_i]_B$. For $1\le i \le c$ we let $(X_{r,s}^i)$ be a $d\times d$ generic matrix of indeterminates. Consider the polynomial ring 
\[
k[X^i_{r,s}\mid 1\le r,s \le d,1\le i\le c]
\]
and let $P$ be the homogeneous ideal generated by the entries of the matrices 
$(X^i_{r,s})(X^j_{r,s})-q(X^j_{r,s})(X^i_{r,s})$, $i<j$, and $(X^i_{r,s})^n$, $1\le i\le c$. Let $V$ denote the affine variety of $P$.  Then any point $p$ in $V$ corresponds to a $A_q^c$-module $M_p$ of dimension $d$, in the sense that the underlying $k$-vector space of $M_p$ is $k^c$ and the matrix $(X^i_{r,s})(p)$ (substitute the coordinates of $p$ in for the $X^i_{r,s}$) represents the linear operator $x_i:M_p\to M_p$, with respect to the standard basis $B$ of $k^c$, for $1\le i\le c$.

Before discussing which $A_q^c$- modules have constant Jordan type, we first consider a weaker condition, namely, that the linear operators $u_\lambda:M\to M$ have constant rank for all $\lambda$. 
For a matrix $B$ we let $I_g(B)$ denote the ideal of $g\times g$ minors of $B$.  Consider now the polynomial ring $k[\Lambda_i, X^i_{r,s} \mid 1\le i\le c, 1\le r,s\le d ]$, in the additional indeterminates 
$\Lambda_i$, and the $d\times d$ matrix  
\[
U_\Lambda=\Lambda_1(X_{r,s}^1)+\cdots+\Lambda_c(X_{r,s}^c)
\]
For $1\le g\le d$ and $p\in V$, we let $U_\Lambda(p)$ denote the matrix, and $I_g(U_\Lambda)(p)$ the ideal of 
$k[\Lambda_1,\dots,\Lambda_c]$, obtained by substituting the coordinates of $p$ in for the $X_{r,s}^i$ in
$U_\Lambda$, and $I_ g(U_\Lambda)$, respectively.

\begin{proposition}\label{cr} For $p\in V$, the linear operators $u_\lambda:M_p\to M_p$ on the 
$A_q^c$-module $M_p$ of dimension $d$ and corresponding to $p$, have constant rank $g$ if and 
only if the following conditions are satisfied.
\begin{enumerate}
\item $I_{g+1}(U_\Lambda)(p)=0$
\item $\sqrt{I_g(U_\Lambda)(p)}\supseteq(\Lambda_1,\dots,\Lambda_c)$
\end{enumerate}
\end{proposition}

\begin{proof} The first condition guarantees that the $(g+1)\times (g+1)$ minors of the matrix $U_\Lambda(p)$
vanish.  This implies that the linear operators $u_\lambda:M_p \to M_p$ have rank at most $g$ for all $\lambda$.
If $g=0$, then the linear operators $u_\lambda:M_p\to M_p$ automatically have constant rank $0$, and by convention $I_0(U_\Lambda)=A^c_q$. Otherwise, the second condition says that the only possible way for $u_\lambda:M_p\to M_p$ to have rank less than $g$ is if 
$\lambda=0$.  Thus for all nonzero $\lambda$, the linear operator $u_\lambda:M_p\to M_p$ has rank $g$.
\end{proof}

Proposition \ref{cr} suggests that the linear operators $u_\lambda:M \to M$ on $A_q^c$-modules of large dimension relative to $c$ tend to have constant rank.  On the other hand, we have the following.

\begin{corollary}
There do not exist linear operators $u_\lambda:M\to M$ of positive constant rank $g$ on $A_q^c$-modules of dimension $d$ when ${d\choose g}^2<c$. Consequently, there are no $A^c_q$-modules of dimension $d$ and of constant Jordan type of rank $g>0$ when ${d\choose g}^2<c$.
\end{corollary}

\begin{proof} Any $A_q^c$-module of dimension $d$ is isomorphic to $M_p$ for some point $p\in V$, and so it suffices to consider modules $M_p$.

There are ${d\choose g}^2$ many $g\times g$ minors of the matrix $U_\Lambda$, and so the ideal $I_g(U_\Lambda)$, being generated by ${d\choose g}^2$ elements,  cannot possibly have radical $(\Lambda_1,\dots,\Lambda_c)$ when 
${d\choose g}^2<c$. Thus there are no $A_q^c$-modules of constant rank $g>0$ and dimension $d$ when 
${d\choose g}^2<c$, and so no such modules of constant Jordan type. 
\end{proof}

For $p\in V$, let us define the $A^c_q$-module $M_p$ to be of \emph{generic rank} $g$ if $I_{g+1}(U_\Lambda)(p)=0$, but $I_g(U_\Lambda)(p)\ne 0$.  Then Proposition \ref{cr} says that among the $p\in V$ for which the $A^c_q$-modules $M_p$ have generic rank $g$, those $p$ corresponding to modules $M_p$ of constant rank $g$ constitute a Zariski open set.  Thus the $A^c_q$-modules $M_p$ generically have constant rank. We want to have a similar statement for modules of constant Jordan type.  This will involve an embellishment of the immediate discussion to the sets $\mathcal U^i_M$ defined above.

Consider the powers $U_\Lambda, U^2_\Lambda,\dots,U^{n-1}_\Lambda$ of the generic matrix $U_\Lambda$.

\begin{definition} 
For $p\in V$, we say that the corresponding $A_q^c$-module $M_p$ has \emph{generic rank} $g=(g_1,\dots,g_{n-1})$ if $I_{g_i+1}(U^i_\Lambda)(p)=0$ and $I_{g_i}(U^i_\Lambda)(p)\ne 0$ for $1\le i \le n-1$.
\end{definition}

\begin{theorem} \label{thm:generic} For $p\in V$, suppose that $M_p$ has generic rank $g=(g_1,\dots,g_{n-1})$.  Then $M_p$ has constant Jordan type if and only if
\[
\sqrt{I_{g_i}(U^i_\Lambda)(p)}=(\Lambda_1,\dots,\Lambda_c)
\]
for $1\le i\le n-1$.
\end{theorem}

\begin{proof} The condition shows that $u^i_\lambda:M_p\to M_p$ is of maximal rank for all nonzero $\lambda$, and all $1\le i\le n-1$. Thus by Lemma \ref{lem:characterizeopen}, $M_p$ has constant Jordan type.
\end{proof}

For fixed $i$, the condition that $p\in V$ satisfies $\sqrt{I_{g_i}(U^i_\Lambda)(p)}=(\Lambda_1,\dots,\Lambda_c)$
corresponds to a Zariski open set of $V$.  Thus the condition of Theorem \ref{thm:generic} corresponds to a finite intersection of Zariski open sets, and thus is open. The question remains when are these open sets nonempty.  The answer is simple: they are nonempty if and only if ${d\choose g_i}^2\ge c$. The upshot of Theorem \ref{thm:generic} is therefore that the $A^c_q$-modules $M_p$ generically have constant Jordan type when $d^2\ge c$.

\begin{remark} The conditions $\sqrt{I_{g_i}(U^i_\Lambda)(p)}=(\Lambda_1,\dots,\Lambda_c)$ seem to depend on the point $p\in V$.  However, if $p'$ is another point of $V$ such that the $A^c_q$-modules $M_p$ and $M_{p'}$ are isomorphic, then there exists a $d\times d$ invertible matrix $E$ such that 
$U_\Lambda(p')=EU^i_\Lambda(p)E^{-1}$, and thus  $I_{g_i}(U^i_\Lambda)(p')=I_{g_i}(U^i_\Lambda)(p)$.
\end{remark}

\section{Auslander-Reiten theory}\label{sec:ARtheory}

In this section, we use Corollary \ref{cor:syzygies} to show that the property of having constant Jordan type is preserved under Auslander-Reiten translates. We then show that if one of the modules in a component of the stable Auslander-Reiten quiver of $A^c_q$ has constant Jordan type, then so do all the other modules in that component.

Recall first that if $A$ is any algebra and $M$ an $A$-module, then from an algebra automorphism $\psi \colon A \to A$ we obtain a new $A$-module ${_{\psi}M}$, called the \emph{twist} of $M$ by $\psi$. The module structure is given by $a \cdot m = \psi (a)m$. The twist commutes with operations such as direct sum and syzygies. We shall be concerned with the homogeneous automorphisms of $A^c_q$, that is, automorphisms which map each generator $x_j$ to a linear combination $\alpha_{1j}x_1 + \cdots + \alpha_{cj}x_c$. Not all automorphisms are of this form. For example, by mapping $x_1$ to $x_1 + x_1^{n-1}x_2^{n-1} \cdots x_c^{n-1}$ and $x_j$ to itself for all $2 \le j \le c$, we have constructed a valid automorphism since all the relations in $A^c_q$ are preserved. If we twist a module having constant Jordan type with such an automorphism, the result may be a module which does not have constant type. However, as the following lemma shows, constant Jordan type is preserved when we twist with homogeneous automorphisms.

\begin{lemma}\label{lem:homogeneoustwist}
If $M$ is an $A^c_q$-module of constant Jordan type, and $\psi \colon A^c_q \to A^c_q$ is a homogeneous automorphism, then the module ${_{\psi}M}$ also has constant Jordan type. Moreover, the Jordan types of $M$ and ${_{\psi}M}$ are the same. 
\end{lemma}

\begin{proof}
For each $1 \le j \le c$ there are scalars $\alpha_{1j}, \dots, \alpha_{cj}$ with
$$x_j \mapsto \alpha_{1j}x_1 + \cdots + \alpha_{cj}x_c$$
The $c \times c$ matrix $E = \left ( \alpha_{ij} \right )$ must have rank $c$; otherwise, there would exist a nonzero $c$-tuple $\beta = ( \beta_1, \dots, \beta_c ) \in k^c$ with $E \beta^T =0$. But then the automorphism $\psi$ would map the nonzero element $\beta_1x_1 + \cdots + \beta_cx_c$ to zero, which is impossible.

Now take a nonzero $c$-tuple $\lambda \in k^c$. Since $\psi ( u_{\lambda} ) = u_{(E \lambda^T)^T}$, the matrix for $u_{\lambda}$ on ${_{\psi}M}$ is the same as the matrix for $u_{(E \lambda^T)^T}$ on $M$. As $E$ has maximal rank, the $c$-tuple $(E \lambda^T)^T$ is nonzero, and so since $M$ has constant Jordan type the matrix for $u_{\lambda}$ on ${_{\psi}M}$ is independent of $\lambda$. Moreover, it is the same as the matrix for $u_{\lambda}$ on $M$.
\end{proof}

\begin{remark}\label{rem:automorphisms}
In most cases, namely when $q$ is not $\pm 1$ (thus $n \ge 3$), the homogeneous automorphisms on $A^c_q$ are actually of a very simple form; they just map each generator $x_j$ to a multiple of itself. To see this, take such an automorphism $\psi \colon A^c_q \to A^c_q$. Then for each $1 \le j \le c$ there are scalars $\alpha_{1j}, \dots, \alpha_{cj}$ with
$$x_j \mapsto \alpha_{1j}x_1 + \cdots + \alpha_{cj}x_c$$
Suppose that $\alpha_{ij}$ is nonzero, and consider another generator $x_s$ for $s \neq j$. If $j < s$, then the relation $x_jx_s - qx_sx_j = 0$ implies that
$$\left (  \alpha_{1j}x_1 + \cdots + \alpha_{cj}x_c \right ) \left (  \alpha_{1s}x_1 + \cdots + \alpha_{cs}x_c \right ) -q \left (  \alpha_{1s}x_1 + \cdots + \alpha_{cs}x_c \right ) \left ( \alpha_{1j}x_1 + \cdots + \alpha_{cj}x_c \right )$$
must be zero. The term involving $x_i^2$ is $(1-q) \alpha_{ij} \alpha_{is} x_i^2$, and so since $(1-q) \alpha_{ij} x_i^2$ is nonzero in $A^c_q$, we see that $\alpha_{is}$ must be zero. The same happens if $s < j$. This shows that if $x_i$ occurs in the linear combination of $\psi (x_j)$, then $x_i$ does \emph{not} occur in the linear combinations of $\psi (x_s)$ for $s \neq j$. Consequently, since $\psi$ is an automorphism, it must simply permute the generators $x_1, \dots, x_c$ up to scalars; there are nonzero scalars $\alpha_1, \dots, \alpha_c$ and a permutation $\sigma \in S_c$ with $\psi ( x_i ) = \alpha_i x_{\sigma (i)}$ for every $1 \le i \le c$.

So far, we have only used that $q \neq 1$. Let us now use the fact that $q \neq \pm 1$ to show that the permutation $\sigma$ must be the identity permutation. If not, there exist two integers $i,j \in \{ 1, \dots, c \}$ with $i < j$ and $\sigma (i) > \sigma (j)$. Then since $x_ix_j = qx_jx_i$ and $x_{\sigma (j)} x_{\sigma (i)} = q x_{\sigma (i)} x_{\sigma (j)}$, we obtain
$$\alpha_i \alpha_j x_{\sigma (i)} x_{\sigma (j)} = \psi \left ( x_ix_j \right ) = \psi \left ( qx_jx_i \right ) = q \alpha_i \alpha_j x_{\sigma (j)} x_{\sigma (i)} = q^2 \alpha_i \alpha_j x_{\sigma (i)} x_{\sigma (j)}$$
but this is impossible when $q \neq \pm 1$. 

To sum up, when $q \neq \pm 1$, then a homogeneous automorphism $\psi \colon A^c_q \to A^c_q$ simply maps $x_i$ to $\alpha_i x_i$ for some nonzero $\alpha_i \in k$. When $q^2 = 1$, however, then there are in general other kinds of homogeneous automorphisms. For example, over any ground field $k$, the quantum complete intersection (and exterior algebra)
$$k \langle x,y \rangle / (x^2, xy+yx, y^2)$$
admits the homogeneous morphisms $(x \mapsto y, y \mapsto x)$ and $(x \mapsto x+y, y \mapsto x-y)$. 
\end{remark}

We can now show that the property of having constant Jordan type is preserved under Auslander-Reiten translates. Recall that for a finite dimensional Frobenius algebra $A$, there is an automorphism $\nu \colon A \to A$, called the \emph{Nakayama automorphism}, with the property that the bimodules $D(A)$ and ${_{\nu}A_1}$ are isomorphic. Here $D(A)$ denotes the vector space dual $\Hom_k(A,k)$ of $A$, and the action on the bimodule ${_{\nu}A_1}$ is defined by $a_1 \cdot a \cdot a_2 = \nu (a_1) a a_2$. The Nakayama automorphism is unique up to inner automorphisms. It is well known that for such an algebra $A$, the Auslander-Reiten translate $\tau M$ of a module $M$ is isomorphic to $\Omega_A^2 \left ( {_{\nu}M} \right )$; see, for example, \cite[Proposition 3.13 and Theorem 8.5]{SkowronskiYamagata}. We can now apply this to our quantum complete intersection $A^c_q$, which is Frobenius.

\begin{theorem}\label{thm:ARtranslate}
An $A^c_q$-module $M$ has constant Jordan type if and only if its Auslander-Reiten translate $\tau M$ does. Moreover, if so, then their stable constant Jordan types are the same.
\end{theorem}

\begin{proof}
By \cite[Lemma 3.1]{Bergh}, the Nakayama automorphism $\nu$ of $A^c_q$ maps each generator $x_i$ to $q^{m_i}x_i$ for some (possibly negative) integer $m_i \in \mathbb{Z}$. Therefore, by Lemma \ref{lem:homogeneoustwist}, the module $M$ has constant Jordan type if and only if ${_{\nu}M}$ does, and with the same Jordan type. It now follows from Corollary \ref{cor:syzygies} that $M$ has constant Jordan type if and only if $\Omega_{A^c_q}^2({_{\nu}M})$ does, and with the same stable Jordan type. Since $\tau M \simeq \Omega_{A^c_q}^2({_{\nu}M})$, the result follows.
\end{proof}

Next, we turn to the stable Auslander-Reiten quiver of $A^c_q$. Our aim is to show that when the ground field $k$ is algebraically closed, then constant Jordan type is a property of the components of the quiver: if one of the modules has constant Jordan type, then so do all the others in that component. We also determine the stable Jordan types of the modules. The key to all this is the fact, proved in \cite{BerghErdmann2}, that when either $n \ge 3$ or $c \ge 3$, then every component of the stable Auslander-Reiten quiver of $A^c_q$ is of the form $\mathbb{Z}A_{\infty}$.

In order to prove this result, we need the following lemma and its corollary. They show that every Auslander-Reiten sequence over $A^c_q$ ending in a module of constant Jordan type splits when we restrict to the subalgebras $k[u_{\lambda}]$.

\begin{lemma}\label{lem:splitsequence}
Let $M$ be an indecomposable $A^c_q$-module of complexity at least $2$, and
$$0 \to \tau M \xrightarrow{f} E \xrightarrow{g} M \to 0$$
the Auslander-Reiten sequence ending in $M$. Then this sequence splits over $k[u_{\lambda}]$ for all nonzero $\lambda \in k^c$.
\end{lemma}

\begin{proof}
Fix a nonzero $\lambda \in k^c$, and denote the algebra $A^c_q$ by just $A$. For every $A$-module $L$, the adjoint isomorphism
$$\Hom_A \left ( A \otimes_{k[u_{\lambda}]} M, L \right ) \to \Hom_{k[u_{\lambda}]} \left ( M, \Hom_A(A,L) \right )$$
together with the natural isomorphism $\Hom_A(A,L) \to L$ of $k[u_{\lambda}]$-modules give an isomorphism
$$\Hom_A \left ( A \otimes_{k[u_{\lambda}]} M, L \right ) \xrightarrow{\varphi_L} \Hom_{k[u_{\lambda}]} \left ( M,L \right )$$
which is natural in $L$. From the map $g$ in the Auslander-Reiten sequence we therefore obtain a commutative diagram
$$\xymatrix{
\Hom_A \left ( A \otimes_{k[u_{\lambda}]} M, E \right ) \ar[r]^{g_*} \ar[d]^{\varphi_E} & \Hom_A \left ( A \otimes_{k[u_{\lambda}]} M, M \right ) \ar[d]^{\varphi_M} \\
\Hom_{k[u_{\lambda}]} \left ( M,E \right ) \ar[r]^{g_*} & \Hom_{k[u_{\lambda}]} \left ( M,M \right )
}$$

Consider now the multiplication map $\mu \colon A \otimes_{k[u_{\lambda}]} M \to M$ given by $a \otimes m \mapsto am$. This is a surjective homomorphism of $A$-modules, and we claim that it cannot split. For if it did, then the $A$-module $M$ would be a direct summand of the $A$-module $A \otimes_{k[u_{\lambda}]} M$. However, the complexity of $M$ as a module over $k[u_{\lambda}]$ is at most one, since all the indecomposable non-projective $k[u_{\lambda}]$-modules are periodic. Moreover, if we take any projective resolution of $M$ over $k[u_{\lambda}]$, and apply $A \otimes_{k[u_{\lambda}]} -$ to it, then the result is a projective resolution of $A \otimes_{k[u_{\lambda}]} M$ over $A$, since $A$ is free as a $k[u_{\lambda}]$-module. Therefore the complexity of the $A$-module $A \otimes_{k[u_{\lambda}]} M$ is at most one. Then $M$ cannot be a direct summand of this module, since the complexity of $M$ is at least $2$.

Since the multiplication map $\mu$ does not split, it factors through the map $g$, so that $\mu = g_* ( \theta )$ for some map $\theta \in \Hom_A \left ( A \otimes_{k[u_{\lambda}]} M, E \right )$. The image of $\mu$ under $\varphi_M$ is the identity on $M$, and so the commutativity of the diagram implies that the map $g$ splits as a homomorphism of $k[u_{\lambda}]$-modules: $1_M = g \circ \varphi_E( \theta )$. This shows that the Auslander-Reiten sequence splits over $k[u_{\lambda}]$.
\end{proof}

It now follows from Theorem \ref{thm:complexitysyzygies} that all the Auslander-Reiten sequences ending in modules of constant Jordan type must split over the subalgebras $k[u_{\lambda}]$.

\begin{corollary}\label{cor:ARsequence}
Every Auslander-Reiten sequence over $A^c_q$ ending in a module of constant Jordan type splits over $k[u_{\lambda}]$ for all nonzero $\lambda \in k^c$.
\end{corollary}

Now we can prove the main result in this section: when the field $k$ is algebraically closed, then the $A^c_q$-modules of constant Jordan type form complete components of the stable Auslander-Reiten quiver. As mentioned, the key ingredient is that in all cases except one, all the components are of the form $\mathbb{Z}A_{\infty}$. In such a component, a module is \emph{quasi-simple} if it belongs to the $\tau$-orbit at the end. For an arbitrary module in the component, there is a shortest sectional path to a quasi-simple module, and the number of modules in such a path is the \emph{quasi-length} of the module. Thus the quasi-simple modules are precisely the ones having quasi-length one.

\begin{theorem}\label{thm:ARcomponents}
Suppose that the ground field $k$ is algebraically closed, and let $\Theta$ be a component of the stable Auslander-Reiten quiver of $A^c_q$ containing a module of constant Jordan type. Then all the modules in $\Theta$ have constant Jordan type. In fact, the following hold:

(1) If $n = c = 2$, then $\Theta$ is of the form $\mathbb{Z} \tilde{A}_{12}$, and its modules are precisely $\{ \Omega_{A^c_q}^i(k) \mid i \in \mathbb{Z} \}$. For even $i$, the stable Jordan type of $\Omega_{A^c_q}^i(k)$ is $[1]$, and for odd $i$ it is $[n-1]$.

(2) If either $n \ge 3$ or $c \ge 3$, then $\Theta$ is of the form $\mathbb{Z} A_{\infty}$. If the stable Jordan type of one of the quasi-simple modules in $\Theta$ is $[1]^{d_1} \cdots [n-1]^{d_{n-1}}$, then every module of quasi-length $l$ has stable Jordan type $[1]^{ld_1} \cdots [n-1]^{ld_{n-1}}$.
\end{theorem}

\begin{proof}
(1) When $n = c = 2$, then our algebra $A^c_q$ is just the commutative truncated polynomial algebra $k[x_1,x_2]/(x_1^2,x_2^2)$, regardless of the characteristic of the field $k$. It is well known (see, for example, the proof of \cite[Lemma II.7.3]{Erdmann}) that the only indecomposable nonprojective and nonperiodic modules over this algebra are the syzygies of the simple module $k$, and that they form a component $\mathbb{Z} \tilde{A}_{12}$ in the stable Auslander-Reiten quiver. By Corollary \ref{cor:syzygies} and the fact that $k$ trivially has constant Jordan type $[1]$, every module in this component has constant stable Jordan type as given. By Theorem \ref{thm:complexitysyzygies}, there are no other indecomposable nonprojective modules of constant Jordan type.

(2) Suppose now that either $n \ge 3$ or $c \ge 3$. By \cite[Theorem 3.6]{BerghErdmann2}, every component of the stable Auslander-Reiten quiver of $A^c_q$ is of the form $\mathbb{Z} A_{\infty}$. Let $M$ be a module in $\Theta$ of constant Jordan type, and suppose that its quasi-length is $l$. Then by Theorem \ref{thm:ARtranslate}, all the modules in $\Theta$ of quasi-length $l$, that is, all the modules in the $\tau$-orbit of $M$, also have constant Jordan type. Moreover, their stable Jordan types are all equal to that of $M$.

Consider the Auslander-Reiten sequence
$$0 \to \tau M \to E \to M \to 0$$
ending in $M$. If $M$ is quasi-simple, that is, if $l = 1$, then $E$ is indecomposable of quasi-length two. Otherwise, the module $E$ is a direct sum $E \simeq E_1 \oplus E_2$ of two indecomposable modules, with $E_1$ of quasi-length $l-1$ and $E_2$ of quasi-length $l+1$. By Corollary \ref{cor:ARsequence}, the sequence splits over $k[u_{\lambda}]$ for all nonzero $\lambda \in k^c$, hence since $M$ and $\tau M$ have constant Jordan type, so does $E$. It now follows from Corollary \ref{cor:directsum} and Theorem \ref{thm:ARtranslate} that every module in $\Theta$ of quasi-length $l+1$, or of quasi-length $l-1$ if $l \ge 2$, must have constant Jordan type. Moreover, the stable Jordan types of the modules of quasi-length $l+1$ are all the same, as are the stable Jordan types of the modules of quasi-length $l-1$. By induction, we see that every module in $\Theta$ must have constant Jordan type. Moreover, two modules having the same quasi-length have the same stable Jordan type. 

Suppose now that $[1]^{d_1} \cdots [n-1]^{d_{n-1}}$ is the stable Jordan type of one of (equivalently, all of) the quasi-simple modules. We show by induction on $l$ that every module $E$ of quasi-length $l$ has stable Jordan type $[1]^{ld_1} \cdots [n-1]^{ld_{n-1}}$, the case $l = 1$ being the quasi-simple case. If $l=2$, then there is an Auslander-Reiten sequence as above, in which $M$ and $\tau M$ are quasi-simple. Then since the sequence splits over $k[u_{\lambda}]$ for all nonzero $\lambda \in k^c$, and the stable Jordan type of both $M$ and $\tau M$ is $[1]^{d_1} \cdots [n-1]^{d_{n-1}}$, we see that the stable Jordan type of $E$ must be $[1]^{2d_1} \cdots [n-1]^{2d_{n-1}}$. If $l \ge 3$, then there is an Auslander-Reiten sequence
$$0 \to \tau M \to E \oplus E' \to M \to 0$$
in which the quasi-length of $E'$ is $l-2$, and that of $M$ and $\tau M$ is $l-1$. By induction, the stable Jordan type of $E'$ is $[1]^{(l-2)d_1} \cdots [n-1]^{(l-2)d_{n-1}}$, and the stable Jordan type of both $M$ and $\tau M$ is $[1]^{(l-1)d_1} \cdots [n-1]^{(l-1)d_{n-1}}$. Using again that the sequence splits over $k[u_{\lambda}]$ for all nonzero $\lambda \in k^c$, we see that the direct sum $E \oplus E'$ must have stable Jordan type $[1]^{2(l-1)d_1} \cdots [n-1]^{2(l-1)d_{n-1}}$, hence the stable Jordan type of $E$ must be $[1]^{ld_1} \cdots [n-1]^{ld_{n-1}}$. This completes the proof.
\end{proof}

\section{Syzygies of the simple module}\label{sec:syzygies}

In this final section, we focus on the syzygies of the simple $A^c_q$-module $k$. Since this module trivially has constant stable Jordan type $[1]$, it follows from Corollary \ref{cor:syzygies} that for every integer $i \in \mathbb{Z}$, the constant stable Jordan type of $\Omega_{A^c_q}^{2i}(k)$ is $[1]$, and for $\Omega_{A^c_q}^{2i+1}(k)$ it is $[n-1]$. Our aim now is to show that these are the \emph{only} modules having these constant stable Jordan types, in the case when $c = 2$, that is, when our algebra has two generators. The arguments we use are to a large extent adaptions of arguments from \cite{Carlson}.

Instead of writing $x_1$ and $x_2$, we shall write $x$ and $y$ for the generators of $A_q^2$, and just $A$ for the algebra $A_q^2$ itself. Thus the algebra we are dealing with throughout this section is
$$A = k \langle x,y \rangle / \left ( x^n, xy-qyx, y^n \right )$$
with $k$, $q$ and $n$ as before. Moreover, as in the previous sections we denote the radical of $A$ by just $\ra$. For an $A$-module $M$, we denote by $\Top M$ its top $M / \ra M$, and by $\Soc M$ its socle, that is, the submodule $\{ m \in M \mid \ra m = 0 \}$. Note that the latter is the same as the set of all elements $m$ in $M$ with $xm = 0$ and $ym = 0$.

For an $A$-module $M$ without any projective summands, consider its minimal complete resolution
$$\cdots \to P_2 \xrightarrow{d_2} P_1 \xrightarrow{d_1} P_0 \xrightarrow{d_0} P_{-1} \xrightarrow{d_{-1}} P_{-2} \to \cdots$$
of free $A$-modules with $\Im d_i \subseteq \ra P_{i-1}$ and $M = \Im d_0$. It is unique up to isomorphism, and so we define $\beta_i(M)$ to be the rank of the free module $P_i$. This is the same as the $k$-vector space dimension of $\Top \Omega_A^i(M)$, since $\Omega_A^i(M) = \Im d_i$. In Section \ref{sec:jordantype}, we defined the complexity of a module over an arbitrary algebra in terms of the dimensions of the projective modules in its minimal projective resolution. Since the algebra we are working over is local, we may just as well use the integers $\beta_i(M)$:
$$\cx M = \inf \{ m \ge 0 \mid \text{there exists } b \in \mathbb{R} \text{ with } \beta_t(M) \le bt^{m-1} \text{ for all } t \ge 0 \}$$

The first result we prove compares $\beta_0(M)$ with $\beta_{-1}(M)$ when the module $M$ has constant stable Jordan type $[1]$ or $[n-1]$. By definition, for all $u_{\lambda} = \lambda_1x + \lambda_2y$ with $\lambda \neq 0$, such a module decomposes over $k[u_{\lambda}]$ into a direct sum $M_i \oplus F$, where $F$ is free and $i$ is either $1$ or $n-1$ (with $M_1$ of dimension $1$, and $M_{n-1}$ of dimension $n-1$). We call the module $M_i$ the nonprojective component.

In the results and proofs to come, we write $k[x]$ and $k[y]$ for the subalgebras of $A$ generated by $x$ and $y$, respectively. These are not to be confused with polynomial rings: the are both isomorphic to $k[z]/(z^n)$.

\begin{proposition}\label{prop:ranks}
Suppose that $M$ is an indecomposable $A$-module of constant stable Jordan type $[1]$ or $[n-1]$. Then the following conditions are equivalent:

(1) $\beta_0(M) > \beta_{-1}(M)$;

(2) There exists a generator $a$ for the nonprojective component of $M$ as a $k[x]$-module, with the following properties: $a \notin \ra M$, and if $M$ is of stable Jordan type $[1]$ then $ya \neq 0$, whereas if $M$ is of type $[n-1]$ then $yx^{n-2}a \neq 0$;

(2') There exists a generator $b$ for the nonprojective component of $M$ as a $k[y]$-module, with the following properties: $b \notin \ra M$, and if $M$ is of stable Jordan type $[1]$ then $xb \neq 0$, whereas if $M$ is of type $[n-1]$ then $y^{n-2}xb \neq 0$;

(3) Any generator for the nonprojective component of $M$ as a $k[x]$-module satisfies (2);

(3') Any generator for the nonprojective component of $M$ as a $k[y]$-module satisfies (2').
\end{proposition}

\begin{proof}
Let $a$ be a generator for the nonprojective component of $M$ as a $k[x]$-module, and consider the surjective map $M/xM \to x^{n-1}M$ induced by multiplication by $x^{n-1}$. Note that both the modules are $A$-modules, since $y$ $q$-commutes with $x$. If we twist the module $x^{n-1}M$ with the automorphism $\psi \colon A \to A$ given by $x \mapsto x$ and $y \mapsto q^{n-1}y$, then the map above induces a homomorphism $M/xM \to {_{\psi}(x^{n-1}M)}$ of $A$-modules. This gives an exact sequence
\begin{equation*}\label{seq:first}
0 \to k \to M/xM \to {_{\psi}(x^{n-1}M)} \to 0 \tag{$\dagger$}
\end{equation*}
over $A$, where the kernel of the map $M/xM \to {_{\psi}(x^{n-1}M)}$ is of dimension one and generated by the coset of $a$. Now consider the socle of $M$ as a $k[x]$-module, i.e.\ $\Soc_{k[x]}M = \{ m \in M \mid xm=0 \}$. Again, since $y$ $q$-commutes with $x$, this is an $A$-module. There is an exact sequence
\begin{equation*}\label{seq:second}
0 \to x^{n-1}M \to \Soc_{k[x]}M \to k \to 0 \tag{$\dagger \dagger$}
\end{equation*}
over $A$, where the one-dimensional module is generated by $a$ if the Jordan type of $M$ is $[1]$, and by $x^{n-2}a$ if the type is $[n-1]$.

Note that in both of the exact sequences, the element $x$ acts trivially on the modules. Consequently, we can view the modules involved as just $k[y]$-modules. For such an $A$-module $L$, the top $L / \ra L$ is just $L/yL$, and so $\dim_k \Top L = \dim_k \Soc_{k[y]} L$. From the sequence (\ref{seq:first}) we therefore obtain
\begin{eqnarray*}
\dim_k \Top M & = & \dim_k \Top M/xM \\
& = & \dim_k \Soc_{k[y]} M/xM \\
& \ge & \dim_k \Soc_{k[y]} {_{\psi}(x^{n-1}M)} \\
& = & \dim_k \Soc_{k[y]} x^{n-1}M \\
& = & \dim_k \Top x^{n-1}M
\end{eqnarray*}
with a strict inequality if and only if the sequence splits. This happens if and only if $a \notin \ra M$, and in this case $\dim_k \Top M = \dim_k \Top x^{n-1}M +1$. On the other hand, if
$$0 \to M \hookrightarrow P \to \Omega_A^{-1}(M) \to 0$$
is exact with $P$ projective, then from the sequence (\ref{seq:second}) we obtain 
\begin{eqnarray*}
\dim_k \Top \Omega_A^{-1}(M) & = & \dim_k \Soc Q \\
& = & \dim_k \Soc M \\
& = & \dim_k \Soc_{k[y]} \left ( \Soc_{k[x]}M \right ) \\
& \ge & \dim_k \Soc_{k[y]} x^{n-1}M \\
& = & \dim_k \Top x^{n-1}M
\end{eqnarray*}
Again, the inequality is strict if and only if the sequence (\ref{seq:second}) splits, which happens if and only if $a \in \Soc M$ when the stable Jordan type of $M$ is $[1]$, and if and only if $x^{n-2}a \in \Soc M$ when the type is $[n-1]$. In other words, the sequence splits if and only if $ya=0$ when the stable Jordan type is $[1]$, and if and only if $yx^{n-2}a =0$ when the type is $[n-1]$. In this case, $\dim_k \Top \Omega_A^{-1}(M) = \dim_k \Top x^{n-1}M +1$. 

From what we have proved, we see that condition (1) holds if and only if $\dim_k \Top M = \dim_k \Top x^{n-1}M +1$ and $\dim_k \Top \Omega_A^{-1}(M) = \dim_k \Top x^{n-1}M$. This happens if and only if (\ref{seq:first}) splits but (\ref{seq:second}) does not, and this is equivalent to condition (2) and (3), since $a$ was an arbitrary generator for the nonprojective component of $M$ as a $k[x]$-module. A completely similar argument shows that (1), (2') and (3') are also equivalent.
\end{proof}

Before we proceed, we make the following definition.

\begin{definition}
Let $M$ be an indecomposable nonprojective $A$-module of stable constant Jordan type either $[1]$ or $[n-1]$. 

(1) The module $M$ satisfies the \emph{rank property with respect to $x$}, abbreviated (RP$x$), if there exists a generator $a$ for the nonprojective component of $M$ over $k[x]$, with the following properties: if the stable Jordan type of $M$ is $[1]$, then $a \notin \ra M$ and $y^{n-1}a \neq 0$, whereas if the stable Jordan type of $M$ is $[n-1]$, then $y^{n-1}x^{n-2}a \neq 0$.

(2) Likewise, $M$ satisfies the \emph{rank property with respect to $y$}, abbreviated (RP$y$), if there exists a generator $b$ for the nonprojective component of $M$ over $k[y]$, with the following properties: if the stable Jordan type of $M$ is $[1]$, then $b \notin \ra M$ and $x^{n-1}b \neq 0$, whereas if the stable Jordan type of $M$ is $[n-1]$, then $y^{n-2}x^{n-1}b \neq 0$.

(3) If $M$ satisfies both (RP$x$) and (RP$y$), then we say that it satisfies the \emph{rank property}, abbreviated (RP).
\end{definition}

\begin{remark}\label{rem:RPgenerator}
If a module $M$ satisfies (RP$x$), then \emph{every} generator for the nonprojective component over $k[x]$ has the defining properties from part (1) of the definition. To see this, suppose that $a$ is as in the definition, and take any other generator $a'$ for the nonprojective component over $k[x]$. If the stable Jordan type of $M$ is $[1]$, then $a' = a + x^{n-1}m$ for some $m \in M$. Then since $a$ does not belong to $\ra M$, neither can $a'$. Moreover, since $y^{n-1}x^{n-1}M =0$, we see that $y^{n-1}a' = y^{n-1}(a+ x^{n-1}m) = y^{n-1}a \neq 0$. If the stable Jordan type of $M$ is $[n-1]$, then $a' = a + xm$ for some $m \in M$. In this case $y^{n-1}x^{n-2}a' = y^{n-1}x^{n-2}(a+xm) = y^{n-1}x^{n-2}a \neq 0$. Thus in both cases the generator $a'$ also has the defining properties from the rank property definition. Similarly, if $M$ satisfies (RP$y$), then every generator for the nonprojective component over $k[y]$ has the defining properties from part (2) of the definition. In the proofs to come, we shall be using these facts without further mention.
\end{remark}

As the terminology suggests, the rank property should of course have something to do with ranks. This is the following result.

\begin{proposition}\label{prop:RPimpliesWRP}
If an indecomposable nonprojective $A$-module satisfies either \emph{(RP$x$)} or \emph{(RP$y$)}, then $\beta_0(M) > \beta_{-1}(M)$.
\end{proposition}

\begin{proof}
Let $M$ be a module that satisfies (RP$x$), and let $a$ be any generator for the nonprojective component of $M$ over $k[x]$. Note first that $M$ is not projective, and so $y^{n-1}x^{n-1}M=0$. If the stable constant Jordan type of $M$ is $[1]$, then by definition of the rank property, the element $a$ is not contained in $\ra M$, and $y^{n-1}a \neq 0$. Therefore $ya$, must be nonzero. If the stable constant Jordan type is $[n-1]$, then by definition $y^{n-1}x^{n-2}a \neq 0$, and in particular $yx^{n-2}a$ must be nonzero. In addition, also in this case the element $a$ cannot belong to $\ra M$: if it did, then we could write $a = xa_1 + ya_2$ for some $a_1, a_2 \in M$, giving $y^{n-1}x^{n-2}a =0$ since $y^n=0$ and $y^{n-1}x^{n-1}M=0$.

We have shown that both in the case the stable constant Jordan type of $M$ is $[1]$, and in the case it is $[n-1]$, condition (3) of Proposition \ref{prop:ranks} holds. It follows that $\beta_0(M) > \beta_{-1}(M)$. A similar proof applies if $M$ satisfies (RP$y$).
\end{proof}

If the stable constant Jordan type of the module is $[n-1]$, then the converse also holds, under some extra conditions. To prove this, we need the first part of the following lemma.

\begin{lemma}\label{lem:multiplicationbyx}
Let $\psi \colon A \to A$ be the automorphism defined by $x \mapsto x$ and $y \mapsto q^{-1}y$, and $\phi \colon A \to A$ the automorphism defined by $x \mapsto qx$ and $y \mapsto y$. Then for every $A$-module $M$ with $\dim_k \sthom_A(M,M)=1$ and ${_{\psi}M} \simeq M \simeq {_{\phi}M}$, and every integer $i$, the following hold.

(1) For every surjective homomorphism $d \colon F \to \Omega_A^i(M)$ with $F$ free, there exist homomorphisms $h_x \colon {_{\psi}\Omega_A^i(M)} \to F$ and $h_y \colon {_{\phi}\Omega_A^i(M)} \to F$, such that the composition $d \circ h_x \colon {_{\psi}\Omega_A^i(M)} \to \Omega_A^i(M)$ is multiplication by $x$, and the composition $d \circ h_y \colon {_{\phi}\Omega_A^i(M)} \to \Omega_A^i(M)$ is multiplication by $y$.

(2) For every injective homomorphism $i \colon \Omega_A^i(M) \to F$ with $F$ free, there exist homomorphisms $h_x \colon F \to {_{\psi^{-1}}\Omega_A^i(M)}$ and $h_y \colon F \to {_{\phi^{-1}}\Omega_A^i(M)}$, such that the composition $h_x \circ i \colon \Omega_A^i(M) \to {_{\psi^{-1}}\Omega_A^i(M)}$ is multiplication by $x$, and the composition $h_y \circ i \colon \Omega_A^i(M) \to {_{\phi^{-1}}\Omega_A^i(M)}$ is multiplication by $y$.
\end{lemma}

\begin{proof}
Consider the module $\Omega_A^i(M)$. When we twist a minimal projective resolution of a module by an automorphism $\theta$, the result is a minimal projective resolution of the corresponding twisted module, hence ${_{\theta}  \Omega_A^i(M)} \simeq  \Omega_A^i( {_{\theta}M} )$. Consequently, there are isomorphisms ${_{\psi}  \Omega_A^i(M)} \simeq \Omega_A^i( {_{\psi}M} ) \simeq \Omega_A^i(M)$ and ${_{\phi}  \Omega_A^i(M)} \simeq \Omega_A^i( {_{\phi}M} ) \simeq \Omega_A^i(M)$, due to the assumptions on $M$. Also, since $\sthom_A( \Omega_A^i(M), \Omega_A^i(M) ) \simeq \sthom_A(M,M)$, the vector space $\sthom_A( \Omega_A^i(M), \Omega_A^i(M) )$ is one-dimensional. This shows that the module $\Omega_A^i(M)$ satisfies the same assumptions as $M$, and so it suffices to prove the lemma for the latter.

For the first part, note that the automorphisms $\psi$ and $\phi$ are precisely the ones we have to twist with in order for multiplication by $x$ to induce an $A$-homomorphism $\mu_x \colon {_{\psi}M} \to M$, and multiplication by $y$ to induce an $A$-homomorphism $\mu_y \colon {_{\phi}M} \to M$
$$\mu_x ( y \cdot m) = \mu_x (q^{-1}ym) = q^{-1}xym = yxm = y \cdot \mu_x(m)$$
$$\mu_y ( x \cdot m) = \mu_y (qxm) = qyxm = xym = x \cdot \mu_y(m)$$
Since $M \simeq {_{\psi}M}$, there is an isomorphism $\sthom_A({_{\psi}M},M) \simeq \sthom_A(M,M)$, and these are one-dimensional vector spaces by assumption. In the stable module category, an isomorphism ${_{\psi}M} \to M$ is nonzero, and so $\mu_x$, which is not an isomorphism, must be zero in $\sthom_A({_{\psi}M},M)$. Thus $\mu_x$ factors through some free $A$-module, and therefore also through $F$. The same argument works for the map $\mu_y$.

The second part is proved similarly. Note that since ${_{\psi}M} \simeq M \simeq {_{\phi}M}$, there are also isomorphisms ${_{\psi^{-1}}M} \simeq M \simeq {_{\phi^{-1}}M}$.
\end{proof}

\begin{proposition}\label{prop:WRPimpliesRP}
Let $M$ be an indecomposable $A$-module of stable constant Jordan type $[n-1]$, and suppose that $\dim_k \sthom_A(M,M)=1$ and ${_{\psi}M} \simeq M \simeq {_{\phi}M}$, where $\psi$ and $\phi$ are the automorphisms from \emph{Lemma \ref{lem:multiplicationbyx}}. Then $M$ satisfies \emph{(RP)} if $\beta_0(M) > \beta_{-1}(M)$.
\end{proposition}

\begin{proof}
We show that $M$ satisfies (RP$x$); the proof that $M$ satisfies (RP$y$) is similar. Let $a$ be any generator for the nonprojective component of $M$ over $k[x]$. To show that $M$ satisfies (RP$x$), we must show that $y^{n-1}x^{n-2}a$ is nonzero. To do this, we consider the projective cover
$$0 \to \Omega_A^1(M) \hookrightarrow P \xrightarrow{d} M \to 0$$
of $M$. Note that $\Omega_A^1(M)$ is also nonprojective and indecomposable, since $M$ is.

View the short exact sequence as a sequence of $k[x]$-modules. As such, the module $M$ is isomorphic to $\Span_k \{ a, xa, \dots, x^{n-2}a \} \oplus F_1$ for some free module $F_1$, where $\Span_k \{ a, xa, \dots, x^{n-2}a \}$ is the nonprojective component of dimension $n-1$. The module $P$ is free, and $\Omega_A^1(M)$ is isomorphic to $k \oplus F_2$ for some free module $F_2$. Consequently, there is an element $p$ of $P$ with $d(p)=a$, and such that $x^{n-1}p$ generates the one-dimensional nonprojective component of $\Omega_A^1(M)$. Note that since $a \in M \setminus \ra M$ by Proposition \ref{prop:ranks}, the element $p$ does not belong to $\ra P$, hence $y^{n-1}x^{n-1}p$ is nonzero. Moreover, since $a$ generates the nonprojective component of $M$ over $k[x]$, the element $x^{n-1}a$ is zero.

Now we use the first part of Lemma \ref{lem:multiplicationbyx}, with the automorphism $\psi$ defined there. Let $h \colon {_{\psi}M} \to P$ be an $A$-homomorphism having the property that the composition $d \circ h \colon {_{\psi}M} \to M$ equals the multiplication map $\mu_x$ given by $x$. Then $d \left ( h(a)-xp \right ) =0$, so that $h(a) = xp + u$ for some $u \in \Omega_A^1(M)$. Note that $x^{n-1}u=0$ since $x^{n-1}u = x^{n-1} \left ( h(a) - xp \right ) = h(x^{n-1}a)$ and $x^{n-1}a =0$, hence $u$ belongs to $xP$. Now since $x^{n-1}p$ generates the nonprojective component of $\Omega_A^1(M)$ over $k[x]$, there is an equality
$$\Omega_A^1(M) \cap xP = \Span_k \{ x^{n-1}p \} + x \Omega_A^1(M)$$
hence we can write $u$ as $u = \alpha x^{n-1}p + xv$ for some scalar $\alpha$ and element $v \in \Omega_A^1(M)$. This gives 
\begin{eqnarray*}
h \left ( y^{n-1}x^{n-2}a \right ) & = & y^{n-1}x^{n-2} h(a) \\
& = & y^{n-1}x^{n-2} \left ( xp + u \right ) \\
& = & y^{n-1}x^{n-2} \left ( xp + \alpha x^{n-1}p + xv \right ) \\
& = & y^{n-1}x^{n-1}p + y^{n-1}x^{n-1}v
\end{eqnarray*}
and so if $y^{n-1}x^{n-2}a$ were zero, then $y^{n-1}x^{n-1}p = - y^{n-1}x^{n-1}v$. As $y^{n-1}x^{n-1}p$ is nonzero, this would imply that $y^{n-1}x^{n-1} \Omega_A^1(M)$ was nonzero, but this is impossible since $\Omega_A^1(M)$ is a nonprojective indecomposable $A$-module. This shows that $y^{n-1}x^{n-2}a$ must be nonzero.
\end{proof}

The idea behind the proof of the main result in this section is to start with a syzygy having the rank property (RP), and then show that its cosyzygies get smaller and smaller until we end up with the module $k$. The first step in this reduction process is the following result.

\begin{proposition}\label{prop:reduce[n-1]}
Let $M$ be an indecomposable $A$-module of stable constant Jordan type $[n-1]$, and suppose that $\dim_k \sthom_A(M,M)=1$ and ${_{\psi}M} \simeq M \simeq {_{\phi}M}$, where $\psi$ and $\phi$ are the automorphisms from \emph{Lemma \ref{lem:multiplicationbyx}}. Then if $M$ satisfies \emph{(RP)}, the module $\Omega_A^{-1}(M)$ is either isomorphic to $k$, or satisfies \emph{(RP)}.
\end{proposition}

\begin{proof}
We show that if $M$ satisfies (RP$x$), then $\Omega_A^{-1}(M)$ is either isomorphic to $k$, or satisfies (RP$x$); the (RP$y$)-version is similar. As in the previous proof, we decompose the module $M$ over $k[x]$ as $\Span_k \{ a, xa, \dots, x^{n-2}a \} \oplus F$ for some free module $F$, where $\Span_k \{ a, xa, \dots, x^{n-2}a \}$ is the nonprojective component of dimension $n-1$. By definition, as $M$ satisfies (RP$x$), the element $y^{n-1}x^{n-2}a$ is nonzero. Now consider the injective envelope
$$0 \to M \hookrightarrow P \xrightarrow{d} \Omega_A^{-1}(M) \to 0$$
of $M$, viewed as a sequence of $k[x]$-modules. Since $x^{n-1}a =0$, we can write $a$ as $a= xp$ for some $p \in P$, with $y^{n-1}x^{n-1}p$ nonzero since $y^{n-1}x^{n-2}a$ is. Thus $p$ does not belong to $\ra P$, and so $d(p) \notin \ra \Omega_A^{-1}(M)$ since $d$ is a projective cover. 

The element $d(p)$ generates the one-dimensional nonprojective component of $\Omega_A^{-1}(M)$ over $k[x]$, hence $xd(p) =0$. If also $yd(p) =0$, then this component is actually a direct summand of $\Omega_A^{-1}(M)$ as an $A$-module. Namely, since $d(p) \notin \ra \Omega_A^{-1}(M)$, there is a maximal submodule $W \subseteq \Omega_A^{-1}(M)$ with $d(p) \notin W$. The submodule generated by $d(p)$ is just $\Span_k \{ d(p) \}$, and since this does not intersect $W$, the $A$-module $\Omega_A^{-1}(M)$ decomposes as a direct sum $W \oplus \Span_k \{ d(p) \}$. But $\Omega_A^{-1}(M)$ is indecomposable over $A$, since $M$ is, and so in this case $\Omega_A^{-1}(M) \simeq k$.

This shows that if $\Omega_A^{-1}(M)$ is not isomorphic to $k$, then $yd(p)$ must be nonzero, that is, the element $yp$ cannot belong to $M$. We must now show that this forces $\Omega_A^{-1}(M)$ to satisfy (RP$x$). As the element $d(p)$ is a generator for the one-dimensional nonprojective component of $\Omega_A^{-1}(M)$ over $k[x]$, and we saw above that $d(p) \notin \ra \Omega_A^{-1}(M)$, it suffices to show that $y^{n-1}d(p)$ is nonzero.

Suppose, for a contradiction, that $y^{n-1}d(p) =0$. Then $y^{n-1}p \in \Ker d =M$. Now we use the second part of Lemma \ref{lem:multiplicationbyx}, with the automorphism $\phi$ defined there. There exists an $A$-homomorphism $h \colon P \to {_{\phi^{-1}}M}$ such that the composition $h \circ i \colon M \to {_{\phi^{-1}}M}$ is multiplication by $y$, where $i$ is the inclusion map $M \hookrightarrow P$. Since $xp$ belongs to $M$, it follows that  $h(xp) = yxp$, and so when we view ${_{\phi^{-1}}M}$ as a submodule of ${_{\phi^{-1}}P}$ we obtain
$$x \cdot h(p) = h(xp) = yxp = q^{-1}xyp = x \cdot yp$$
Then $x \cdot \left ( h(p) - yp \right ) =0$ in ${_{\phi^{-1}}P}$, so with $u = h(p) - yp$ we obtain $h(p) = yp + u$ with $x \cdot u =0$. From the above, the element $y^{n-1}p$ belongs to $M$, hence $y^{n-1}h(p) = h(y^{n-1}p) = y^np=0$. The expression $h(p) = yp + u$ then gives $y^{n-1} \cdot u =0$ in ${_{\phi^{-1}}P}$.

Since both $x \cdot u$ and $y^{n-1} \cdot u$ are zero in ${_{\phi^{-1}}P}$, the element $u$ can be written as $yx^{n-1}v$ for some element $v \in P$. Now consider the element $p' = p + x^{n-1}v$ in $P$. It does not belong to $\ra P$, since $p \notin \ra P$. Moreover, $yp' = yp + u = h(p)$, hence $yp'$ belongs to $M$. Finally, $xp' = xp =a$. The latter means that we can start the argument all over again, but this time with $p'$ instead of $p$. However, this time $yp'$ belongs to $M$, and this is impossible when $\Omega_A^{-1}(M)$ is not isomorphic to $k$.
\end{proof}

In the proof of the second and final step in the reduction process, we need the following lemma. Recall that an indecomposable nonprojective $A$-module $M$ of stable constant Jordan type $[1]$ satisfies (RP$x$) if the following hold: there exists a generator $a$ for the nonprojective component of $M$ over $k[x]$, such that $a \notin \ra M$ and $y^{n-1}a \neq 0$. By Remark \ref{rem:RPgenerator}, every such generator $a$ has these properties. Similarly, the module $M$ satisfies (RP$y$) if there exists a generator $b$ for the nonprojective component of $M$ over $k[y]$, with $b \notin \ra M$ and $x^{n-1}b \neq 0$. The lemma shows that when the assumptions from the last couple of results hold, then if we take \emph{any} element $m \in M \setminus \ra M$, either $x^{n-1}m$ or $y^{n-1}m$ must be nonzero.

\begin{lemma}\label{lem:generator[1]}
Let $M$ be an indecomposable $A$-module of stable constant Jordan type $[1]$, satisfying either \emph{(RP$x$)} or \emph{(RP$y$)}. Moreover, suppose that $\dim_k \sthom_A(M,M)=1$ and ${_{\psi}M} \simeq M \simeq {_{\phi}M}$, where $\psi$ and $\phi$ are the automorphisms from \emph{Lemma \ref{lem:multiplicationbyx}}. Then for every element $m \in M \setminus \ra M$, either $x^{n-1}m$ or $y^{n-1}m$ is nonzero.
\end{lemma}

\begin{proof}
We prove the (RP$x$)-version. Take such an element $m \in M \setminus \ra M$, and suppose, for a contradiction, that $x^{n-1}m=0$ and $y^{n-1}m=0$. Consider a projective cover
$$0 \to \Omega_A^1(M) \hookrightarrow P \xrightarrow{d} M \to 0$$
of $M$, and let $p \in P$ be an element with $m = d(p)$. Since $m \notin \ra M$, the element $p$ does not belong to $\ra P$, and so $y^{n-1}x^{n-1}p \neq 0$. Note that since $0 = y^{n-1}d(p) = d ( y^{n-1}p)$, the element $y^{n-1}p$ belongs to $\Ker d = \Omega_A^1(M)$, and similarly so does $x^{n-1}p$.

We shall modify the element $p$ and obtain a new element $p' \in P$, with the property that $d(p') \in M \setminus \ra M$, and with
$$y^{n-1}p' = y^{n-1}p \hspace{5mm} \text{and} \hspace{5mm} xd(p')=0$$
It then follows that $y^{n-1}x^{n-1}M$ is nonzero, a contradiction since $M$ is indecomposable and not projective. To see this, let $z \in P$ be an element with the property that $d(z)$ generates the one-dimensional nonprojective component of $M$ over $k[x]$. Then $d(p') = \alpha d(z) + w$ for some scalar $\alpha \in k$ and some element $w$ in the $k[x]$-free summand. As $xd(p') =0 = xd(z)$, we see that $xw =0$, so $w = x^{n-1}w'$ for some $w'$ in the free summand. This implies that $w \in x^{n-1} M$, and since $d(p') \notin \ra M$, the scalar $\alpha$ must be nonzero. Now
$$0 = y^{n-1}m = y^{n-1}d(p) = d( y^{n-1}p) = d( y^{n-1}p') = y^{n-1}d(p') = \alpha y^{n-1} d(z) + y^{n-1}w$$
and so since $y^{n-1} d(z)$ is nonzero by (RP$x$), we see that $y^{n-1}x^{n-1}w' = y^{n-1}w \neq 0$. Thus we have produced a nonzero element in $y^{n-1}x^{n-1}M$.

To construct the element $p'$ with the desired properties, we apply the second part of Lemma \ref{lem:multiplicationbyx} to the module $\Omega_A^1(M)$: there exists an $A$-homomorphism $h \colon P \to {_{\psi^{-1}}\Omega_A^1(M)}$ such that the composition $h \circ i \colon \Omega_A^1(M) \to {_{\psi^{-1}}\Omega_A^1(M)}$ is multiplication by $x$, where $i$ is the inclusion map $\Omega_A^1(M) \hookrightarrow P$. Therefore, in ${_{\psi^{-1}}\Omega_A^1(M)}$, there are equalities
$$q^{n-1}y^{n-1}h(p) = y^{n-1} \cdot h(p) = h( y^{n-1}p) = xy^{n-1}p = q^{n-1}y^{n-1}xp$$
and this is a nonzero element since $y^{n-1}x^{n-1}p \neq 0$ and $x$ and $y$ $q$-commute. Viewing ${_{\psi^{-1}}\Omega_A^1(M)}$ as a submodule of ${_{\psi^{-1}}P}$, we see that
$$q^{n-1}y^{n-1} \left ( h(p) - xp \right ) = 0$$
in the latter, and then also in $P$ itself. Therefore, there is an element $u \in yP$ with $h(p) - xp =u$; say $u = yu'$ for some $u' \in P$.

Now we use the fact that $x^{n-1}p \in \Omega_A^1(M)$, and obtain
$$x^{n-1}u = x^{n-1} \left ( h(p) - xp \right ) = x^{n-1}h(p) = h( x^{n-1}p) = h \circ i ( x^{n-1}p ) = x^np=0$$
Thus $x^{n-1}yu' =0$ in $P$, so we can write $u' = xu_1 + y^{n-1}u_2$ for some $u_1,u_2 \in P$. This, in turn, gives
$$u = yu' = yxu_1 = x (q^{-1}yu_1) = xv$$
where $v = q^{-1}yu_1 \in yP$. Now consider the element $p + v$ in $P$; call it $p'$. It follows from the definition of $u$ and $v$ that
$$xp' = xp+xv = xp+u = h(p)$$
and
$$y^{n-1}p' = y^{n-1}p + y^{n-1}v = y^{n-1}p$$
since $v \in yP$. Moreover, since $y^{n-1}p$ belongs to $\Omega_A^1(M)$, so does $y^{n-1}p'$. Note also that $p'$ cannot belong to $\ra P$, for $p \notin \ra P$ whereas $v = q^{-1}yu_1 \in \ra P$. Therefore, since $d$ is a projective cover, the element $d(p')$ does not belong to $\ra M$. Moreover, as $xp'$ equals $h(p)$, and the latter belongs to ${_{\psi^{-1}}\Omega_A^1(M)}$ and therefore also to $\Omega_A^1(M) = \Ker d$, we see that $xd(p') = d(xp') =0$. We have shown that the element $p'$ has the desired properties.
\end{proof}

We can now prove the second and final step in the reduction process.

\begin{proposition}\label{prop:reduce[1]}
Suppose that $n \ge 3$, and let $M$ be an indecomposable $A$-module of stable constant Jordan type $[1]$. Furthermore, suppose that $\dim_k \sthom_A(M,M)=1$ and ${_{\psi}M} \simeq M \simeq {_{\phi}M}$, where $\psi$ and $\phi$ are the automorphisms from \emph{Lemma \ref{lem:multiplicationbyx}}. Then if $M$ satisfies \emph{(RP)}, so does $\Omega_A^{-1}(M)$.
\end{proposition}

\begin{proof}
As before, we only show that $\Omega_A^{-1}(M)$ satisfies (RP$x$), since the (RP$y$)-version is completely similar. However, this time we need the fact that the module $M$ satisfies \emph{both} (RP$x$) and (RP$y$).

As a $k[x]$-module, we can decompose $M$ as $\Span_k \{ a \} \oplus F$ for some free module $F$, where $a$ generates the one-dimensional nonprojective component. The rank property (RP$x$) then gives $a \notin \ra M$ and $y^{n-1}a \neq 0$. As in the previous proof, consider the injective envelope
$$0 \to M \hookrightarrow P \xrightarrow{d} \Omega_A^{-1}(M) \to 0$$
of $M$, viewed as a sequence of $k[x]$-modules. Since $xa =0$, we can write $a$ as $a= x^{n-1}p$ for some $p \in P$, and then since $y^{n-1}x^{n-1}p = y^{n-1}a \neq 0$, this element $p$ does not belong to $\ra P$. Then $d(p)$ does not belong to $\ra \Omega_A^{-1}(M)$, and it generates the nonprojective component of dimension $n-1$ over $k[x]$. By definition of the rank property (RP$x$), we must show that $y^{n-1}x^{n-2}d(p)$ is nonzero.

Suppose, for a contradiction, that $y^{n-1}x^{n-2}d(p) =0$. Then $y^{n-1}x^{n-2}p \in \Ker d =M$; call this element $m$. We shall show that $m = y^{n-1}m'$ for some other element $m' \in M$. To do this, note first that both $ym$ and $x^2m$ are zero. Now consider $M$ as a module over $k[y]$; as such, it decomposes as $\Span_k \{ b \} \oplus F'$ for some free module $F'$, where $b$ generates the one-dimensional nonprojective component. Since $ym=0$, we can write $m$ as $m= \alpha b + y^{n-1}m'$ for some scalar $\alpha$ and some $m' \in F'$. Furthermore, since $M$ satisfies (RP$y$), the element $x^{n-1}b$ is nonzero. Therefore, if $\alpha$ is nonzero, then so is $x^{n-1}m$, since $x^{n-1}y^{n-1}m' =0$ as $M$ is not projective. But this cannot be the case, for we saw above that $x^2 m=0$, and by assumption $n \ge 3$. The scalar $\alpha$ must therefore be zero, and this shows that $m = y^{n-1}m'$.

Finally, consider the element $x^{n-2}p - m'$ in $P$; call it $p'$. Since $m' \in M$ and $x^{n-1}p =a \in M$, the element $xp'$ belongs to $M$; let us denote it by $m''$. This element does not belong to $\ra M$, for $m'' = a + xm'$, and $a \notin \ra M$ since $M$ satisfies (RP$x$). Therefore, by Lemma \ref{lem:generator[1]}, either $y^{n-1}m''$ or $x^{n-1}m''$ must be nonzero. As $xa=0$, it must be the case that $y^{n-1}m'' \neq 0$, but
$$y^{n-1}m'' = y^{n-1}x^{n-1}p - y^{n-1}xm' = q^{1-n} x \left ( y^{n-1}x^{n-2}p - y^{n-1}m' \right ) = q^{1-n}  x ( m-m)=0$$
This is a contradiction, and so $y^{n-1}x^{n-2}d(p)$ must be nonzero.
\end{proof}

We now prove the main result in this section. It characterizes the syzygies of the simple $A$-module $k$. As mentioned, the idea behind the proof is to start with a syzygy having the rank property (RP), and then show that its cosyzygies get smaller and smaller, using Proposition \ref{prop:RPimpliesWRP}, Proposition \ref{prop:reduce[n-1]}  and Proposition \ref{prop:reduce[1]}. We then end up with the module $k$ itself. 

\begin{theorem}\label{thm:syzygiesofk}
If the ground field $k$ is infinite, then for an indecomposable nonprojective $A$-module $M$, the following are equivalent:

(1) $M \simeq \Omega_A^i(k)$ for some $i \in \mathbb{Z}$;

(2) $M$ has constant stable Jordan type either $[1]$ or $[n-1]$. Moreover, $\sthom_A(M,M)$ is one-dimensional, and ${_{\psi}M} \simeq M$ for every homogeneous automorphism $\psi \colon A \to A$;

(3) $M$ has constant stable Jordan type either $[1]$ or $[n-1]$. Moreover, $\sthom_A(M,M)$ is one-dimensional, and ${_{\psi}M} \simeq M \simeq {_{\phi}M}$, where $\psi$ and $\phi$ are the automorphisms from \emph{Lemma \ref{lem:multiplicationbyx}}.
\end{theorem}

\begin{proof}
Suppose that (1) holds. As we explained in the proof of Lemma \ref{lem:generator[1]}, whenever we take an $A$-module $X$ and automorphism $\psi \colon A \to A$, there is an isomorphism ${_{\psi}\Omega_A^1(X)} \simeq \Omega_A^1( {_{\psi}X} )$. Thus for every integer $i$, the modules ${_{\psi}\Omega_A^i(X)}$ and $\Omega_A^i( {_{\psi}X} )$ are isomorphic. This implies that
$${_{\psi}M} \simeq {_{\psi}\Omega_A^i(k)} \simeq \Omega_A^i( {_{\psi}k} ) \simeq \Omega_A^i(k) \simeq M$$
for every homogeneous automorphism $\psi$, since ${_{\psi}k} \simeq k$. Moreover, from the isomorphisms
$$\sthom_A(M,M) \simeq \sthom_A( \Omega_A^i(k), \Omega_A^i(k) ) \simeq \sthom_A(k,k)$$
we see that $\dim_k \sthom_A(M,M) =1$, and from Corollary \ref{cor:syzygies} we know that $M$ has constant stable Jordan type either $[1]$ or $[n-1]$. This shows that (2) holds, and since (2) trivially implies (3), we need to show that (3) implies (1). 

Suppose therefore that (3) holds, and take any syzygy $\Omega_A^i(M)$ of $M$. As above, all the properties from (3) also hold for this module: the vector space $\sthom_A( \Omega_A^i(M), \Omega_A^i(M) )$ is one-dimensional, there are isomorphisms ${_{\psi}\Omega_A^i(M)} \simeq \Omega_A^i(M) \simeq {_{\phi}\Omega_A^i(M)}$, and $\Omega_A^i(M)$ has constant stable Jordan type either $[1]$ or $[n-1]$. We shall show that there exists a syzygy of Jordan type $[n-1]$ satisfying the rank property (RP).

As mentioned before Theorem \ref{thm:complexitysyzygies}, the algebra $A$ has finitely generated cohomology. This means that its Hochschild cohomology ring $\HH^*(A)$ is Noetherian, and that $\Ext_A^*(X,Y)$ is a finitely generated $\HH^*(A)$-module for all $A$-modules $X$ and $Y$. As explained in \cite[Section 2]{BerghErdmann}, it follows from \cite[Corollary 3.5]{Oppermann2} that there exists a polynomial subalgebra $H = k[ \eta_1, \eta_2 ]$ of $\HH^*(A)$, with both $\eta_1$ and $\eta_2$ in degree $2$, and such that $\Ext_A^*(X,Y)$ is a finitely generated $H$-module for all $A$-modules $X$ and $Y$. In particular, $\Ext_A^*(M, A/ \ra )$ is finitely generated over $H$.

Since the ground field $k$ is infinite, the first part of the proof of \cite[Theorem 2.5]{Bergh1} shows that there exists a homogeneous element $\eta \in H$, of degree $2$, such that multiplication
$$\Ext_A^i(M, A/ \ra ) \xrightarrow{\cdot \eta} \Ext_A^{i+2}(M, A/ \ra )$$
is injective for $i \gg 0$. From \cite[Proposition 2.2]{Bergh1} and its proof, we see that this element $\eta$ gives rise to an eventually surjective chain map $f_{\eta} \colon  \mathbb{P} \to \mathbb{P}$ of degree $-2$, where $\mathbb{P}$ is the minimal projective resolution 
$$\cdots \to P_2 \xrightarrow{d_2} P_1 \xrightarrow{d_1} P_0$$
of $M$. For each $i \ge 0$, this chain map gives a commutative diagram
$$\xymatrix{
P_{i+3} \ar[r]^{d_{i+3}} \ar[d]^{f_{i+3}} & P_{i+2} \ar[r] \ar[d]^{f_{i+2}} & \Omega_A^{i+2}(M) \ar[r] & 0 \\
P_{i+1} \ar[r]^{d_{i+1}} & P_i \ar[r] & \Omega_A^i(M) \ar[r] & 0
}$$ 
with exact rows, a diagram that induces a map $g_{i+2} \colon \Omega_A^{i+2}(M) \to \Omega_A^i(M)$. As $f_i$ is surjective for all $i \gg 0$, so is $g_i$.

By Theorem \ref{thm:complexitysyzygies}, the complexity of the module $M$ is $2$. In particular, this means that $M$ is not periodic, and so the maps $g_i$ cannot be isomorphisms. Namely, if $\Omega_A^{i+2}(M)$ were isomorphic to $\Omega_A^i(M)$, then $\Omega_A^2(M)$ would be isomorphic to $M$. Thus for all $i \gg 0$, the map $g_{i+2}$ is surjective but not injective, and therefore $\dim_k \Omega_A^{i+2}(M) > \dim_k \Omega_A^i(M)$. Taking the alternating sum of the dimensions of the modules in the exact sequence
$$0 \to \Omega_A^{i+2}(M) \to P_{i+1} \xrightarrow{d_{i+1}} P_i \to \Omega_A^i(M) \to 0$$
we then see that $\dim_k P_{i+1} > \dim_k P_i$. This shows that $\beta_{i+1}(M) > \beta_i(M)$ for all $i \gg 0$. Now since every second syzygy of $M$ is of constant stable Jordan type $[n-1]$, we may now apply Proposition \ref{prop:WRPimpliesRP}: there exists an integer $i$ such that $\Omega_A^{i}(M)$ is of constant stable Jordan type $[n-1]$ and satisfies (RP). 

Let $i_0$ be an integer with the property that $\Omega_A^{i_0}(M)$ is of constant stable Jordan type $[n-1]$ and satisfies (RP), and such that $\beta_{i_0}(M)$ is minimal; by the above, such an integer exists. By Proposition \ref{prop:reduce[n-1]}, the module $\Omega_A^{i_0-1}(M)$ is either isomorphic to $k$, or satisfies (RP). We claim that it must be isomorphic to $k$. If not, then by Proposition \ref{prop:reduce[1]}, the module $\Omega_A^{i_0-2}(M)$ also satisfies (RP), since the stable Jordan type of $\Omega_A^{i_0-1}(M)$ is $[1]$. However, when we apply Proposition \ref{prop:RPimpliesWRP} to both $\Omega_A^{i_0}(M)$ and $\Omega_A^{i_0-1}(M)$, we see that there are strict inequalities $\beta_{i_0}(M) > \beta_{i_0-1}(M) > \beta_{i_0-2}(M)$. Since the module $\Omega_A^{i_0-2}(M)$ is of constant stable Jordan type $[n-1]$, this contradicts the minimality of $\beta_{i_0}(M)$. Consequently, the module $\Omega_A^{i_0-1}(M)$ must be isomorphic to $k$.
\end{proof}


\begin{thebibliography}{EHSST}
\bibitem[Ben]{Benson}D.\ Benson, \emph{Representations of elementary abelian $p$-groups and vector bundles}, Cambridge Tracts in Mathematics, 208, Cambridge University Press, Cambridge, 2017, xvii+328 pp.
\bibitem[BEH]{BensonErdmannHolloway}D.\ Benson, K.\ Erdmann, M.\ Holloway, \emph{Rank varieties for a class of finite-dimensional local algebras}, J.\ Pure Appl.\ Algebra 211 (2007), no.\ 2, 497--510.
\bibitem[Be1]{Bergh1}P.A.\ Bergh, \emph{Complexity and periodicity}, Colloq.\ Math.\ 104 (2006), no.\ 2, 169--191.
\bibitem[Be2]{Bergh}P.A.\ Bergh, \emph{Ext-symmetry over quantum complete intersections}, Arch.\ Math.\ (Basel) 92 (2009), no.\ 6, 566--573.
\bibitem[BE1]{BerghErdmann}P.A.\ Bergh, K.\ Erdmann, \emph{The Avrunin-Scott theorem for quantum complete intersections}, J.\ Algebra 322 (2009), no.\ 2, 479--488.
\bibitem[BE2]{BerghErdmann2}P.A.\ Bergh, K.\ Erdmann, \emph{The stable Auslander-Reiten quiver of a quantum complete intersection}, Bull.\ Lond.\ Math.\ Soc.\ 43 (2011), no.\ 1, 79--90.
\bibitem[BeO]{BerghOppermann}P.A.\ Bergh, S.\ Oppermann, \emph{Cohomology of twisted tensor products}, J.\ Algebra 320 (2008), no.\ 8, 3327--3338.
\bibitem[Car]{Carlson}J.F.\ Carlson, \emph{Endo-trivial modules over ($p,p$)-groups}, Illinois J.\ Math.\ 24 (1980), no.\ 2, 287--295.
\bibitem[CFP]{CarlsonFriedlanderPevtsova}J.F.\ Carlson, E.M.\ Friedlander, J.\ Pevtsova, \emph{Modules of constant Jordan type}, J.\ Reine Angew.\ Math.\ 614 (2008), 191--234. 
\bibitem[Erd]{Erdmann}K.\ Erdmann, \emph{Blocks of tame representation type and related algebras}, Lecture Notes in Mathematics, 1428, Springer-Verlag, Berlin, 1990, xvi+312 pp.
\bibitem[EHSST]{EHSST}K.\ Erdmann, M.\ Holloway, N.\ Snashall, {\O}.\ Solberg, R.\ Taillefer, \emph{Support varieties for selfinjective algebras}, K-Theory 33 (2004), no.\ 1, 67--87.
\bibitem[Op1]{Oppermann}S.\ Oppermann, \emph{A lower bound for the representation dimension of $kC_n^p$}, Math.\ Z.\ 256 (2007), no.\ 3, 481--490.
\bibitem[Op2]{Oppermann2}S.\ Oppermann, \emph{Hochschild cohomology and homology of quantum complete intersections}, Algebra Number Theory 4 (2010), no.\ 7, 821--838.
\bibitem[SkY]{SkowronskiYamagata}A.\ Skowro{\'n}ski, K.\ Yamagata, \emph{Frobenius algebras. I. Basic representation theory}, EMS Textbooks in Mathematics, European Mathematical Society (EMS), Z{\"u}rich, 2011, xii+650 pp.
\end{thebibliography}
\end{document}